\documentclass[11pt]{amsart}
\usepackage{amsmath,amssymb,mathrsfs,color}
\usepackage{color}
\def\R{{\mathbb R}}
\newtheorem{lemma}{Lemma}[section]
\newtheorem{theorem}[lemma]{Theorem}
\newtheorem{remark}[lemma]{Remark}
\newtheorem{prop}[lemma]{Proposition}
\newtheorem{coro}[lemma]{Corollary}
\newtheorem{definition}[lemma]{Definition}
\newtheorem{example}[lemma]{Example}
\allowdisplaybreaks
\numberwithin{equation}{section}
\begin{document}

\title[Poisson stable solutions for SDEs with L\'evy noise]
{Poisson stable solutions for stochastic differential equations with L\'evy noise}

%% First author
\author{Xin Liu}
\address{X. Liu : School of Mathematical Sciences,
Dalian University of Technology, Dalian 116024, P. R. China}
\email{liuxin17@mail.dlut.edu.cn}

%% Second author
\author{Zhenxin Liu}
\address{Z. Liu (Corresponding author): School of Mathematical Sciences,
Dalian University of Technology, Dalian 116024, P. R. China}
\email{zxliu@dlut.edu.cn}

%\thanks{This work is partially supported by NSFC Grants 11271151, 11522104, and the startup and
%Xinghai Youqing funds from Dalian University of Technology.}

%%%%%%%%%%%%%%%%%%%%%%%%%%%%%%%%%%%
%\date{\today}
\date{September 30, 2019}
 %Secondary: 35B35, 35B40.

\begin{abstract}

In this paper, we use a unified framework to study Poisson stable (including
stationary, periodic, quasi-periodic, almost periodic, almost automorphic, Birkhoff recurrent,
almost recurrent in the sense of Bebutov, Levitan almost periodic, pseudo-periodic, pseudo-recurrent
and Poisson stable) solutions for semilinear stochastic differential equations driven by infinite
dimensional L\'evy noise with large jumps. Under suitable conditions on drift, diffusion and jump
coefficients, we prove that there exist solutions which inherit the Poisson stability of coefficients.
Further we show that these solutions are globally asymptotically stable in square-mean sense. Finally,
we illustrate our theoretical results by several examples.

\end{abstract}
\maketitle

\section{Introduction}
The notion of Poisson stability was first introduced by Poincar\'e in his famous work \cite{Poincare}
in the late 19th century. He pointed out
that in all bounded Hamilton systems the orbits of aperiodic solutions are stable in the sense of Poisson.
Poisson stable motions in dynamical systems generally include stationary, periodic, quasi-periodic,
almost periodic \cite{Bohr_1925,Bohr_1925_2,Bohr_1926,Boh_1933}, almost automorphic
\cite{Boh_1955,B62,Veech_1963,SY}, Birkhoff recurrent \cite{Bir},
Levitan almost periodic \cite{Lev,Lev_1953}, almost recurrent \cite{Bebu,Sch65},
pseudo-periodic \cite{Bohr_1947}, pseudo-recurrent \cite{Sch62,Sch72} etc.
Note that in the literature on topological dynamics, Poisson stable motions are sometimes
called recurrent motions, especially in discrete situations. It is well-known that recurrence
is one of central topics of both dynamical systems and probability theory, which describes
the asymptotic behavior and complexity of dynamical systems and
Markov processes. By Poincar\'e recurrence theorem and Birkhoff recurrence theorem,
recurrence exists extensively in dynamical systems. On the other hand, it is known that (positive) recurrence
is essentially equivalent to the existence of invariant (probability) measures for Markov processes.

As mentioned above, recurrence has been extensively studied for deterministic dynamical systems.
Now let us recall some studies on recurrent solutions for stochastic differential equations (SDEs) which are
closely relevant to our present work. There are many works on recurrent solutions to the continuous SDE based on Gaussian noise,
such as Khasminskii \cite{Khasminskii}, Zhao and Zheng \cite{ZZ} for periodicity (see also Chen et al \cite{CHLY} and Ji et al \cite{JQSY}
for the study of periodicity in the framework of Fokker-Planck equations), Halanay \cite{Halanay},
Morozan and Tudor \cite{Morozan and Tudor},
Da Prato and Tudor \cite{DT}, Arnold and Tudor \cite{AT}, Liu and Wang \cite{LW}, Li et al \cite{LLW} for almost periodicity,
Fu and Liu \cite{Fu and Liu}, Chen and Lin \cite{CLin} for almost automorphy, and Cheban and Liu \cite {CL}, Cheng and Liu \cite{CmL}
for general recurrence, among others.

At the same time, similar researches for SDEs based on semimartingales with jumps have been underway,
but it turns out that results are not simply parallel to those continuous SDEs. When referring to discontinuous
and fluctuation cases, we usually consider L\'evy processes with jumps since the analysis to L\'evy processes
is much simpler and they can also provide valuable information. L\'evy processes that a class of
c\`adl\`ag processes possess stationary and independent increments, including Brownian motions and
Poisson processes as special cases, form an important subclass of both semimartingales and Markov processes.
It is because L\'evy processes keep independent increment property of Brownian motions but not strictly
request continuity of orbits, permitting break points exist in time-continuous financial, biological models etc,
L\'evy models with jumps can better depict the phenomena in these fields.
In fact, it is fair to say that SDEs with jumps driven by L\'evy processes are
quite useful and they have been applied in mathematical finance, biology and other areas.
Here we just mention some works directly related to ours, without any claim of completeness.
The monographs Sato \cite{Ken-iti}, Applebaum \cite{A} and Peszat and Zabczyk \cite{PZ} respectively
introduced the theory of L\'evy processes, the theory of SDEs driven
by finite dimensional L\'evy processes and the infinite dimensional case systematically.
Wang and Liu \cite{WL} and Liu and Sun \cite{LS} respectively proved the existence of almost
periodic and almost automorphic solutions to SDEs perturbed by L\'evy noise.

Motivated by the work of Cheban and Liu \cite{CL}
on Poisson stable solutions of SDEs with Gaussian noise, we intend to apply comparable (strongly comparable)
methods of B. A. Shcherbakov to investigate Poisson stable solutions of SDEs based on L\'evy noise
with large jumps in the form:
\begin{align*}
{\rm d }Y(t)=&(AY(t)+f(t,Y(t))){\rm d }t+g(t,Y(t)){\rm d }W(t) \\
&+\int_{|x|_U<1}F(t,Y(t-),x)\widetilde{N}({\rm d }t,{\rm d }x)+\int_{|x|_U\geq1}G(t,Y(t-),x)N({\rm d }t,{\rm d }x),
\end{align*}
where the operator $A$ generates an exponentially stable semigroup and
$f,g,F,G$ are Poisson stable functions in $t$.
We wonder whether or not there exist Poisson stable solutions for the above SDE,
since intuitively large jumps may destroy the Poisson stability of solutions. For this
interesting question, we elaborate in this paper that under some conditions on coefficients $f,g,F,G$,
there always exists a unique bounded solution which have the same character of recurrence as
coefficients in distribution sense and the solution is globally asymptotically stable. Besides
we emphasize the fact that the Poisson stability of solutions is determined by the weakest
Poisson stability of coefficients $f,g,F,G$. It coincides with our cognition and we show
this point by several examples.

Our paper is organized as follows. Section 2 begins with definitions of Poisson stable functions,
L\'evy processes and their basic properties. We simply
review L\'evy-It\^o decomposition and B. A. Shcherbakov's comparable (strongly comparable)
methods by character of recurrence. In Section 3, we prove (Theorem \ref{thmone}) that under
some suitable conditions, there exist bounded solutions, taking on the Poisson stability of
coefficients in distribution sense, for linear SDEs with L\'evy noise. In Section 4, we prove
the semilinear SDE case and it is the main result of this paper (Theorem \ref{thmtwo}).
In Section 5, we discuss the asymptotic stability of Poisson stable solutions.
In Section 6, we illustrate the application of our theoretical results by some examples.

\section{Preliminaries}
\subsection{The space $C(\mathbb R, \mathcal Y)$}
Let $(\mathcal Y,\rho)$ be a complete metric space. Denote by $C(\mathbb R,\mathcal Y)$
the space of all continuous functions $\varphi :\mathbb R
\to \mathcal Y$ equipped with the metric
\begin{equation*}\label{eqD1}
d(\varphi,\psi):=\sup\limits_{k>0}\min\{\max\limits_{|t|\le
k}\rho(\varphi(t),\psi(t)),k^{-1}\}.
\end{equation*}
Note that $(C(\mathbb R,\mathcal Y),d)$ is a complete metric space. It is known that for any
$\varphi,\psi\in C(\mathbb R, \mathcal Y)$, $d(\varphi,\psi) = \varepsilon$ if and only if
\begin{equation*}
\max\limits_{|t|\le \varepsilon^{-1}}\rho(\varphi(t),\psi(t))=\varepsilon
\end{equation*}
(see, e.g. \cite[ChI]{Sch72}, \cite{Sch85,sib}).

\begin{remark}\label{remD1} \rm
1. The metric $d$ generates the compact-open topology on $C(\mathbb R,\mathcal Y)$.

2. The following statements are equivalent.
\begin{enumerate}
\item $d(\varphi_n,\varphi)\to 0$ as $ n\to \infty$.
\item For each $k>0$, $\lim\limits_{n\to \infty}\max\limits_{|t|\le
k}\rho(\varphi_n(t),\varphi(t))=0$.
\item There exists a sequence $l_n\to +\infty$ such that $\lim\limits_{n\to \infty}\max\limits_{|t|\le
l_n}\rho(\varphi_n(t),\varphi(t))=0$.
\end{enumerate}
\end{remark}

\subsection{Poisson stable functions}
Let us recollect the types of Poisson stable functions to be studied in this paper and the relations among them.
We recommend the reader to \cite{Sel_71,Sch72,Sch85,sib} for further details.
\begin{definition}\label{hull}\rm
For given $\varphi\in C(\mathbb R,\mathcal Y )$, $\varphi^h$ denotes the
{\em $h$-translation of $\varphi$}, i.e. $\varphi^h(t):=\varphi(h+t)$ for $t\in\mathbb R$. $H(\varphi)$
denotes the hull of $\varphi$, which is the set of all the limits of $\varphi^{h_n}$ in $C(\mathbb R,\mathcal Y)$, i.e.
$H(\varphi):=\{\psi\in C(\mathbb R, \mathcal Y):\psi=\lim\limits_{n\to\infty}
\varphi^{h_n} \hbox{ for some sequence } \{h_n\} \subset \mathbb R\}$.
\end{definition}

\begin{remark} \rm
It is well-known that the mapping $\pi: \mathbb R\times C(\mathbb R,\mathcal Y)\to C(\mathbb R, \mathcal Y)$
defined by $\pi(h,\varphi) =\varphi^h$ is a dynamical system, i.e. $\pi(0,\varphi)=\varphi$,
$\pi(h_1+h_2,\varphi)=\pi(h_2,\pi(h_1,\varphi))$ and
the mapping $\pi$ is continuous (see, e.g. \cite{Ch2015,Sel_71}). In particular, the mapping $\pi$
restricted to $\mathbb R\times H(\varphi)$ is a dynamical system.
\end{remark}

\begin{definition} \rm
A number $\tau\in\mathbb R$ is said to be {\em $\varepsilon$-shift} for $\varphi \in C(\mathbb R,\mathcal Y)$ if $d(\varphi^{\tau},\varphi)<\varepsilon$.
\end{definition}

\begin{definition} \rm
A function $\varphi \in C(\mathbb R,\mathcal Y)$ is called {\em Poisson stable in the positive (respectively, negative) direction}
if for every $\varepsilon >0$ and $l>0$ there exists $\tau >l$ (respectively, $\tau<-l$) such that $d(\varphi^{\tau},\varphi)<\varepsilon$.
The function $\varphi$ is called {\em Poisson stable} provided it is Poisson stable in both directions.
\end{definition}

\begin{definition} \rm
A function $\varphi\in C(\mathbb R,\mathcal Y)$ is called {\em stationary} (respectively, {\em $\tau$-periodic})
if $\varphi(t)=\varphi(0)$ (respectively, $\varphi(t+\tau)=\varphi(t)$) for all $t\in \mathbb R$.
\end{definition}
\begin{definition} \rm
A function $\varphi \in C(\mathbb R,\mathcal Y)$ is called {\em quasi-periodic with the spectrum of frequencies $\nu_1,\nu_2,\ldots,\nu_m$}
if the following conditions are fulfilled.
\begin{enumerate}
\item The numbers $\nu_1,\nu_2,\ldots,\nu_m$ are rationally independent.
\item There exists a continuous function $\Phi :\mathbb R^{m}\to \mathcal Y$ such that
$\Phi(t_1+2\pi,t_2+2\pi,\ldots,t_m+2\pi)=\Phi(t_1,t_2,\ldots,t_m)$ for all $(t_1,t_2,\ldots,t_m)\in \mathbb R^{m}$.
\item $\varphi(t)=\Phi(\nu_1 t,\nu_2 t,\ldots,\nu_m t)$ for all $t\in \mathbb R$.
\end{enumerate}
\end{definition}

\begin{definition} \rm
Let $\varepsilon >0$. A number $\tau \in \mathbb R$ is called {\em $\varepsilon$-almost period} of the function
$\varphi$ if $\rho(\varphi(t+\tau),\varphi(t))<\varepsilon$ for all $t\in\mathbb R$.
Denote by $\mathcal T(\varphi,\varepsilon)$
the set of $\varepsilon$-almost periods of $\varphi$.
\end{definition}

\begin{definition} \rm
A function $\varphi \in C(\mathbb R,\mathcal Y)$ is said to be {\em almost periodic} if the set of
$\varepsilon$-almost periods of $\varphi$ is {\em relatively dense} for each $\varepsilon >0$,
i.e. for each $\varepsilon >0$ there exists $l=l(\varepsilon)>0$ such that
$\mathcal T(\varphi,\varepsilon)\cap [a,a+l]\not=\emptyset$ for all $a\in\mathbb R$.
\end{definition}

\begin{definition} \rm
A function $\varphi \in C(\mathbb R,\mathcal Y)$ is said to be {\em pseudo-periodic in the positive}
(respectively, {\em negative}) {\em direction} if for each $\varepsilon >0$ and
$l>0$ there exists a $\varepsilon$-almost period $\tau >l$
(respectively, $\tau <-l$) of the function $\varphi$. The function
$\varphi$ is called pseudo-periodic if it is pseudo-periodic in both
directions.
\end{definition}

\begin{definition}\label{defPR}\rm %(\cite{Sch68,Sch72,Sch85})
A function $\varphi \in C(\mathbb
R,\mathcal Y)$ is called  {\em pseudo-recurrent} if for any
$\varepsilon >0$ and $l\in\mathbb R$ there exists a constant $k\ge l$ such that for any $\tau_0\in \mathbb R$
we can find a number $\tau \in [l,k]$
satisfying
$$
\sup\limits_{|t|\le {\varepsilon}^{-1}}\rho(\varphi(t+\tau_0
+\tau),\varphi(t+\tau_0))\le \varepsilon.
$$
\end{definition}

\begin{remark}\rm
The inclusion relations among the above recurrent functions are stationary $\Rightarrow$ periodic $\Rightarrow$
quasi-periodic $\Rightarrow$ almost periodic $\Rightarrow$ pseudo-periodic $\Rightarrow$ pseudo-recurrent $\Rightarrow$
Poisson stable, where $\Rightarrow$ means implying.
\end{remark}

\begin{definition}\rm
A function $\varphi \in C(\mathbb R,\mathcal Y)$ is called {\em almost automorphic} if and only if for any sequence $\{t'_{n}\} \subset
\mathbb R$ there are a subsequence $\{t_n\}$ and some function
$\psi: \mathbb R\to \mathcal Y$ such that
\begin{equation*}
\varphi(t+t_n)\to \psi(t)\ \ \mbox{and}\ \ \psi(t-t_n)\to \varphi(t)
\end{equation*}
uniformly in $t$ on every compact subset from $\mathbb R$.
\end{definition}
In what follows, we denote $(\mathcal X,\gamma)$ as a complete metric space.
\begin{definition} \rm
A function $\varphi\in C(\mathbb R,\mathcal Y)$ is called {\em Levitan almost periodic} if there exists
an almost periodic function $\psi \in C(\mathbb R,\mathcal X)$ such that for any $\varepsilon >0$ there exists $\delta =\delta (\varepsilon)>0$
such that $d(\varphi^{\tau},\varphi)<\varepsilon$ for all $\tau \in \mathcal T(\psi,\delta)$.
\end{definition}

\begin{definition} \rm
A function $\varphi \in C(\mathbb R,\mathcal Y)$ is called {\em almost recurrent (in the sense of Bebutov)} if for every
$\varepsilon >0$ the set $\{\tau : d(\varphi^{\tau},\varphi)<\varepsilon\}$ is relatively dense.
\end{definition}

\begin{definition} \rm
A function $\varphi\in C(\mathbb R,\mathcal Y)$ is called {\em Lagrange stable} if $\{\varphi^{h}: h\in \mathbb R\}$
is a relatively compact subset of $C(\mathbb R,\mathcal Y)$.
\end{definition}

\begin{definition} \rm
A function $\varphi \in C(\mathbb R,\mathcal Y)$ is called {\em Birkhoff recurrent} if it is almost recurrent and Lagrange stable.
\end{definition}

\begin{remark}\label{remPR} \rm (\cite{Sch68,Sch72,Sch85,sib})
\begin{enumerate}
\item
The inclusion relations among the above recurrence notions are almost periodicity $\Rightarrow$ almost automorphy $\Rightarrow$
Levitan almost periodicity $\Rightarrow$ almost recurrence $\Rightarrow$ Poisson stability.
\item
Every almost automorphic function is Birkhoff recurrent and every Birkhoff recurrent function is pseudo-recurrent, but the inverse
statement is not true in general.

\end{enumerate}
\end{remark}
Finally, we remark that a Lagrange stable function is not necessary Poisson stable, but all other types of functions
introduced above are Poisson stable.

\subsection{Shcherbakov's comparability method by character of recurrence}

Let $\varphi \in C(\mathbb R,\mathcal Y)$. Denote by $\mathfrak N_{\varphi}$ (respectively,
$\mathfrak M_{\varphi}$) the family of all sequences
$\{t_n\}\subset \mathbb R$ such that $\varphi^{t_n} \to \varphi$ (respectively,
$\{\varphi^{t_n}\}$ converges) in $C(\mathbb R,\mathcal Y)$ as $n\to \infty$. We write $\mathfrak N_{\varphi}^{u}$
(respectively, $\mathfrak M_{\varphi}^{u}$) to mean the family of sequences $\{t_n\}\in \mathfrak N_{\varphi}$
(respectively, $\{t_n\}\in \mathfrak M_{\varphi}$)
such that $\varphi^{t_n}$ converges to $\varphi$
(respectively,  $\{\varphi^{t_n}\}$ converges) uniformly in $t\in\mathbb R$ as $n\to \infty$.

\begin{definition} \rm
A function $\varphi \in C(\mathbb R,\mathcal Y)$ is said to be {\em comparable (by character of recurrence)} with $\psi\in C(\mathbb R,\mathcal X)$
if $\mathfrak N_{\psi}\subseteq \mathfrak N_{\varphi}$;
$\varphi$ is said to be {\em strongly comparable (by character of recurrence)} with $\psi$ if $\mathfrak M_{\psi}\subseteq \mathfrak M_{\varphi}$.
\end{definition}

\begin{theorem}\label{th1}{\rm(\cite[ChII]{Sch72}, \cite{scher75}, \cite{CL})}
Let $\varphi\in C(\mathbb R,\mathcal Y)$, $\psi\in C(\mathbb R,\mathcal X)$. Then the following statements hold.
\begin{enumerate}
\item $\mathfrak M_{\psi}\subseteq \mathfrak M_{\varphi}$ implies $\mathfrak N_{\psi}\subseteq \mathfrak N_{\varphi}$, and hence
      strong comparability implies comparability.
\item $\mathfrak M_{\psi}^{u} \subseteq \mathfrak M_{\varphi}^{u}$ implies $\mathfrak N_{\psi}^{u}\subseteq \mathfrak N_{\varphi}^{u}$.

\item  Let $\varphi \in C(\mathbb R,\mathcal Y)$ be comparable with $\psi\in C(\mathbb R,\mathcal X)$.
 If the function $\psi$ is stationary (respectively, $\tau$-periodic, Levitan almost periodic, almost recurrent, Poisson stable), then so is $\varphi$.

\item Let $\varphi \in C(\mathbb R,\mathcal Y)$ be strongly comparable with $\psi\in C(\mathbb R,\mathcal X)$.
If the function $\psi$ is quasi-periodic with the spectrum of frequencies
$\nu_1$, $\nu_2$, $\dots$, $\nu_m$ (respectively, almost periodic, almost automorphic, Birkhoff recurrent, Lagrange stable), then so is $\varphi$.

\item Let $\varphi \in C(\mathbb R,\mathcal Y)$ be strongly comparable with $\psi\in C(\mathbb R,\mathcal X)$
and $\psi$ be Lagrange stable. If $\psi$ is pseudo-periodic (respectively, pseudo-recurrent), then so is $\varphi$.
\end{enumerate}
\end{theorem}

\subsection{L\'evy process}\label{levy}
Throughout the paper, we fix a complete probability space $(\Omega, \mathcal{F}, \mathbf{P} )$
and assume that $(\mathbb{H}, |\cdot|)$ and
$(U,|\cdot|_U)$ are real separable Hilbert spaces. Denote by $(L(U,\mathbb{H}),\|\cdot\|_{L(U,\mathbb H)})$
the space of all bounded linear operators from $U$ to $\mathbb{H}$. The L\'evy processes we consider are $U$-valued.
We now review the definition of L\'evy processes and
the important L\'evy-It\^o decomposition theorem; for more details, see \cite{A,K}.

\begin{definition}\rm
A $U$-valued stochastic process $L=(L(t),t\ge 0)$ is called
\emph{L\'{e}vy process} if it has the following three properties:
\begin{enumerate}
  \item $L(0)=0$ almost surely.
  \item $L$ has stationary and independent increments,
  i.e. the law of $L(t+h)-L(t)$ does not depend on $t$ and for all $0\leq t_1<t_2<...<t_n<\infty$, random variables
  $L(t_1)$, $L(t_2)-L(t_1)$,..., $L(t_n)-L(t_{n-1})$ are independent.
  \item $L$ is \emph{stochastically continuous}.
  i.e. for all $\epsilon>0$ and for all $s>0$
   \[
     \lim_{t\rightarrow s}\mathbf P(|L(t)-L(s)|_U>\epsilon)=0.
  \]
\end{enumerate}
\end{definition}

Since a L\'evy process $L$ is c\`adl\`ag, the associated \emph{jump process}
$\Delta L=(\Delta L(t), t\geq 0)$ is given by $\Delta
L(t)=L(t)-L(t-)$. For any Borel set $B$ in
$U-\{0\}$, define the random counting measure
\[
   N(t,B)(\omega):=\sharp \{0\leq s\leq t: \Delta L(s)(\omega)\in B\}
   = \sum_{0\leq s\leq t} \chi_ B (\Delta L(s)(\omega)),
\]
where $\chi_ B$ is the indicator function of $B$. We call $\nu(\cdot):=\mathbb{E}(N(1,\cdot))$ the {\em intensity measure}
of $L$. We say that a Borel set $B$ in
$U-\{0\}$ is {\em bounded below} if $0\notin \overline{B}$, the
closure of $B$. If $B$ is bounded below, then $N(t,B)<\infty$ holds almost
surely for all $t\geq 0$. For given $B$, $(N(t,B), t\geq 0)$ obeys Poisson distribution
with intensity $\nu (B)$, so $(N(t,B), t\ge 0)$ is a
Poisson process and $N$ is called {\em Poisson random measure}.
For each $t\geq 0$ and $B$ bounded below, we define
the \emph{compensated Poisson random measure} by
\[
\widetilde{N}(t,B)=N(t,B)-t\nu(B).
\]

\begin{prop}[L\'evy-It\^o decomposition]\label{lid}
A $U$-valued L\'{e}vy process $L$ can be represented as
\begin{equation}\label{levy-ito}
L(t)=at+W(t)+ \int_{|x|_U <1} x \widetilde N (t,{\rm d }x)+ \int_{|x|_U
\ge 1} x N (t,{\rm d}x),
\end{equation}
where $a\in U$ and $W$ is a $U$-valued $\mathcal Q$-Wiener process. $N$ is a Poisson random measure
on $\R^+ \times (U- \{0\})$ with intensity $\nu$, which is independent of $W$.
Here the intensity measure $\nu$ satisfies
\begin{equation}\label{nu}
\int_U(|x|_U^2\wedge1)\nu({\rm d } x)<\infty
\end{equation}
and $\widetilde N$ is the compensated Poisson random measure of $N$.
\end{prop}
As for $\mathcal Q$-Winner processes and the stochastic integral based on them, the monograph \cite{DZ}
provides a thorough description. Assume that $L_1$ and $L_2$ are two independent, identically distributed L\'evy processes with
decompositions as in Proposition \ref{lid} and let
\[
L (t) =\left\{ \begin{array}{ll}
                 L_1(t), &  \hbox{ for } t\ge 0, \\
                 -L_2(-t), &  \hbox{ for } t\leq 0.
               \end{array}
 \right.
\]
Then $L$ is a two-sided L\'evy process. In this paper, we consider
 two-sided L\'evy process $L$ which is defined on the filtered probability
space $(\Omega,\mathcal F,\mathbf P,(\mathcal F_t)_{t\in\mathbb
R})$ and suppose that the covariance operator $\mathcal Q$ of $W$ is of trace class, i.e. ${\rm Tr} \mathcal Q < \infty$.

\begin{remark}\label{rnu} \rm
It follows from \eqref{nu} that $\int_{|x|_U\ge 1} \nu({\rm d } x)<
\infty$. For convenience, we set hereafter
\[
b:=\int_{|x|_U\ge 1} \nu({\rm d } x).
\]
\end{remark}

\begin{remark}\label{Lp}\rm
Note that the stochastic process $\widetilde L=(\widetilde L(t),
t\in \mathbb R)$ given by $\widetilde L(t): = L(t+ s) -L(s)$ for
some $s\in \mathbb R$ is also a two-sided L\'evy process which
shares the same law as $L$. In particular, when $s\in\mathbb R^+$,
the similar conclusion holds for one-sided L\'evy processes.
\end{remark}

\subsection{Stochastic differential equations with L\'evy noise}
Consider the following SDE driven by L\'evy noise
\begin{equation}\label{L1}
{\rm d } Y(t)=(AY(t)+f(t,Y(t))){\rm d }t+g(t,Y(t)){\rm d }L(t),
\end{equation}
where $A$ is an infinitesimal generator which generates a $\mathcal C^0$-semigroup
$\{T(t)\}_{t\geq0}$ on $\mathbb{H}$; $f:\mathbb R\times \mathbb H\to\mathbb H$,
$g:\mathbb R\times \mathbb H\to L(U,\mathbb H)$; $L$ is a $U$-valued L\'evy process.
When the L\'evy process $L$ is $p$-integrable $(p\ge 1)$, i.e. $\mathbb E|L(t)|^p<\infty$
for all $t\in\R$, $L$ has the following form of L\'evy-It\^o decomposition
\begin{equation*}
L(t)=at+W(t)+\int_U x \widetilde N(t,{\rm d }x).
\end{equation*}
Then, we only need to consider the SDE of the form
\begin{equation*}
{\rm d }Y(t)=(AY(t)+f(t,Y(t))){\rm d }t+g(t,Y(t)){\rm d }W(t)+\int_U F(t,Y(t-),x) \widetilde N({\rm d }t,{\rm d }x).
\end{equation*}
In this paper, we do not assume that the L\'{e}vy process $L$ is $p$-integrable. Then by
L\'evy-It\^o decomposition \eqref{levy-ito}, equation (\ref{L1}) reads
\begin{align}\label{fullform}
{\rm d }Y(t)=&(AY(t)+f(t,Y(t))){\rm d }t+g(t,Y(t)){\rm d }W(t) \\
&+\int_{|x|_U<1}F(t,Y(t-),x)\widetilde N({\rm d }t,{\rm d }x)+\int_{|x|_U\geq1}G(t,Y(t-),x)N({\rm d }t,{\rm d }x),\nonumber
\end{align}
where $F$ and $G$ are $\mathbb H$-valued. It allows us to study large jumps with considerable probability.
We set $\mathcal F_t:=\sigma\{L(u):u\leq t\}$.

\begin{definition}\label{four one}\rm
An $\mathcal{F}_t$-adapted process $\{Y(t)\}_{t\in\R}$ is called a {\em mild solution} of (\ref{fullform})
if it satisfies the corresponding stochastic integral equation
\begin{align}\label{SDEA}
 Y(t)=&T(t-r)Y(r)+\int_r^t T(t-s)f(s,Y(s)){\rm d }s+\int_r^t T(t-s)g(s,Y(s)){\rm d }W(s) \\
 &+\int_r^t\int_{|x|_U<1}T(t-s)F(s,Y(s-),x)\widetilde{N}({\rm d }s,{\rm d }x) \nonumber\\
 &+\int_r^t\int_{|x|_U\geq1}T(t-s)G(s,Y(s-),x)N({\rm d }s,{\rm d }x), \nonumber
\end{align}
for all $t\geq r$ and each $r\in \mathbb R$.
\end{definition}

\subsection{Compatible (strong compatible) solutions in distribution}

Let $\mathcal P(\mathbb H)$ be the space of all Borel probability measures on $\mathbb H$ endowed
with the $\beta$ metric:
$$
\beta (\mu,\nu) :=\sup\left\{ \left| \int f {\rm d } \mu - \int f {\rm d }\nu\right|: \|f\|_{BL} \le 1
\right\}, \quad \hbox{for }\mu,\nu\in \mathcal P(\mathbb H).
$$
Here $f$ varies in the space of bounded Lipschitz continuous real-valued functions on $\mathbb H$ with the norm
\[
\|f\|_{BL}= Lip(f) + \|f\|_\infty,
\]
where
\[
 Lip(f)=\sup_{x\neq y} \frac{|f(x)-f(y)|}{|x-y|},~~ \|f\|_{\infty}=\sup_{x\in \mathbb H}|f(x)|.
\]
A sequence $\{\mu_n\}\subset \mathcal P(\mathbb H)$ is
said to {\em weakly converge} to $\mu$ if $\int f {\rm d}\mu_n\to \int f {\rm d}\mu$
for all $f\in C_b(\mathbb H)$, where $C_b(\mathbb H)$ is the space
of all bounded continuous real-valued functions on $\mathbb H$. As we know that
$(\mathcal P(\mathbb H),\beta)$ is a separable complete metric space and that a sequence $\{\mu_n\}$
weakly converges to $\mu$ if and only if $\beta(\mu_n, \mu) \to0$ as $n\to\infty$.
See \cite[\S 11.3]{Dudley} for $\beta$ metric and related properties.

\begin{definition}\label{aad}\rm
A sequence of random variables $\{x_n\}$ is said to \emph{converge in distribution} to the random variable $x$ if the corresponding laws
$\{\mu_n\}$ of $\{x_n\}$ weakly converge to the law $\mu$ of $x$, i.e. $\beta(\mu_n,\mu)\to 0$.
\end{definition}

\begin{definition} \rm
Let $\{\varphi (t)\}_{t\in\mathbb R}$ be a mild solution of \eqref{fullform}. Then $\varphi$ is called
{\em compatible} (respectively, {\em strongly compatible}) {\em in distribution} if
$\mathfrak N_{(f,g,F,G)} \subseteq \tilde{\mathfrak N}_{\varphi}$
(respectively, $\mathfrak M_{(f,g,F,G)} \subseteq \tilde{\mathfrak M}_{\varphi}$),
where $\tilde{\mathfrak N}_{\varphi}$ (respectively, $\tilde{\mathfrak M}_{\varphi}$)
means the set of all sequences $\{t_n\}\subset\mathbb R$
such that the sequence $\{\varphi(\cdot+t_n)\}$ converges to $\varphi(\cdot)$
(respectively, $\{\varphi(\cdot+t_n)\}$ converges) in distribution uniformly on any compact interval.
\end{definition}

\subsection{The function space $BUC(\mathbb R\times \mathcal Y,\mathcal X)$}
\begin{definition}\label{defQ1} \rm
A function $f:\mathbb R\times \mathcal Y\to \mathcal X$ is called {\em continuous at $t_0\in\mathbb R$ uniformly w.r.t. $y\in
Q$} if for any $\varepsilon >0$ there exists $\delta
=\delta(t_0,\varepsilon)>0$ such that $|t-t_0|<\delta$ implies
$\sup\limits_{y\in Q}\gamma(f(t,y),f(t_0,y))<\varepsilon$. The
function $f:\mathbb R\times \mathcal Y\to \mathcal X$ is called {\em continuous on
$\mathbb R$ uniformly w.r.t. $y\in Q$} if it is continuous at every
point $t_0\in\mathbb R$ uniformly w.r.t. $y\in Q$.
\end{definition}

\begin{remark}\label{BUC}\rm
Note that if $Q$ is a compact subset of $\mathcal Y$ and $f:\mathbb R\times \mathcal Y\to \mathcal X$ is
a continuous function, then $f$ is continuous on $\mathbb R$ uniformly w.r.t. $y\in Q$.
\end{remark}

Denote by $BUC(\mathbb R\times \mathcal Y,\mathcal X)$ the set of all functions
$f:\mathbb R\times \mathcal Y\to \mathcal X$ possessing the following properties:
\begin{enumerate}
\item $f$ are continuous in $t$ uniformly w.r.t. $y$ on every bounded subset $Q\subseteq \mathcal Y$.
\item $f$ are bounded on every bounded subset from $\mathbb R\times \mathcal Y$.
\end{enumerate}

Let $f,g\in BUC(\mathbb R\times \mathcal Y, \mathcal X)$ and $\{Q_n\}$ be a sequence of bounded subsets
from $\mathcal Y$ such that $Q_n\subset Q_{n+1}$ for any
$n\in\mathbb N$ and $\mathcal Y=\bigcup\limits_{n\ge 1} Q_n$.
We equip the function space $BUC(\mathbb R\times \mathcal Y, \mathcal X)$ with the metric
\begin{equation}\label{eqQD1}
d_{BUC}(f,g):=\sum_{n=1}^{\infty}\frac{1}{2^n}\frac{d_n(f,g)}{1+d_{n}(f,g)},
\end{equation}
where $d_{n}(f,g):=\sup\limits_{|t|\le n,\ y\in Q_n}\gamma(f(t,y),g(t,y))$.
Then $(BUC(\mathbb R\times \mathcal Y, \mathcal X),d_{BUC})$ is a complete metric space and $d_{BUC}(f_n,f)\to 0$
if and only if $f_n(t,y)\to f(t,y)$ uniformly w.r.t. $(t,y)$ on every bounded subset from $\mathbb R\times \mathcal Y$.
For given $f\in BUC(\mathbb R\times \mathcal Y,\mathcal X)$ and $\tau\in\mathbb R$,
denote the {\em translation of $f$} by $f^\tau$, i.e. $f^\tau(t,y):=
f(t+\tau,y)$ for $(t,y)\in \mathbb R\times \mathcal Y$, and the {\em hull of
$f$} by $H(f):=\overline{\{f^\tau:\tau\in\mathbb R\}}$ with the
closure being taken under the metric $d_{BUC}$ given by \eqref{eqQD1}.
Note that the mapping $\pi: \mathbb R\times BUC(\mathbb R\times \mathcal Y,\mathcal X)
\to BUC(\mathbb R\times \mathcal Y,\mathcal X)$ defined by $\pi(\tau,f):= f^\tau$
is a dynamical system, i.e. $\pi(0,f) = f$,
$\pi(\tau_1+\tau_2,f) = \pi(\tau_2,\pi(\tau_1,f))$ and the
mapping $\pi$ is continuous. See \cite[\S 1.1]{Ch2015} or \cite{Sel_71} for details.

We use $BC(\mathcal Y,\mathcal X)$ to denote the set of all
continuous and bounded on every bounded subset $Q\subset \mathcal Y$ functions $f:\mathcal Y\to \mathcal X$
and let
$$
d_{BC}(f,g):=\sum_{n=1}^{\infty}\frac{1}{2^n}\frac{d_n(f,g)}{1+d_{n}(f,g)}
$$
for any $f,g\in BC(\mathcal Y,\mathcal X)$, where $d_{n}(f,g):=\sup\limits_{y\in Q_n}\gamma(f(y),g(y))$
with $Q_n$ the same as in \eqref{eqQD1}. Then $(BC(\mathcal Y,\mathcal X),d_{BC})$ is a complete metric space.

Let now  $f\in BUC(\mathbb R\times \mathcal Y,\mathcal X)$ and $\mathfrak f :\mathbb R\to BC(\mathcal Y,\mathcal X)$ a
mapping defined by equality $\mathfrak f(t):=f(t,\cdot)$.

\begin{remark}\label{remBUC} \rm
It is not hard to check that:
\begin{enumerate}
\item  $\mathfrak M_{f}=\mathfrak
M_{\mathfrak f}$ for any $f\in BUC(\mathbb R\times \mathcal Y,\mathcal X)$;

\item $\mathfrak M_{f}^{u}=\mathfrak
M_{\mathfrak f}^{u}$ for any $f\in BUC(\mathbb R\times \mathcal Y,\mathcal X)$.
\end{enumerate}
Here $\mathfrak M_{f}$ is the set of all sequences $\{t_n\}$ such that
$f^{t_n}$ converges in the space $BUC(\mathbb R\times \mathcal Y,\mathcal X)$ and $\mathfrak M_{f}^{u}$ is
the set of all sequences $\{t_n\}$ such that $f^{t_n}$ converges
in the space $BUC(\mathbb R\times \mathcal Y,\mathcal X)$ uniformly w.r.t. $t\in\mathbb R$.
\end{remark}

\section{Linear equations}
Let $\mathfrak B$ be a Banach space with the norm $\|\cdot\|_{\mathfrak B}$.
Denote by $C_{b}(\mathbb R,\mathfrak B)$ the Banach space of all continuous and bounded mappings $\varphi :\mathbb R\to \mathfrak
B$ endowed with the norm $\|\varphi\|_{\infty}:=\sup\{\|\varphi(t)\|_{\mathfrak B}:t\in\mathbb R\}$.

\begin{remark}\label{rem hull}\rm
If $f\in C_{b}(\mathbb R,\mathfrak B)$, then for any $\widetilde f \in H(f)$ we have
$\|\widetilde f(t)\|_\mathfrak B \leq\|f\|_\infty$ for each $t\in\mathbb R$.
Here $H(f)$ means the hull of $f$ defined in Definition \ref{hull}.
\end{remark}

We define for $p\geq 2$
\[
\mathcal L^p(\mathbf{P}; \mathbb{H}):=\mathcal{L}^{p}(\Omega, \mathcal{F}, \mathbf{P}; \mathbb{H})
=\bigg\{Y:\Omega\to \mathbb H \bigg| \mathbb E|Y|^p=\int_{\Omega}|Y|^p {\rm d }\mathbf{P}<\infty\bigg\}
\]
with the norm
\[
\|Y\|_{\mathcal L^p(\mathbf{P}; \mathbb{H})}:=\left(\int_{\Omega}|Y|^p {\rm d }{\mathbf P}\right)^\frac{1}{p}.
\]
Note that $(\mathcal L^p(\mathbf{P}; \mathbb{H}),\|\cdot\|_{\mathcal L^p(\mathbf{P}; \mathbb{H})})$ is a Banach space. Set
\begin{align*}
\mathcal L^2(\mathbf P;L(U,\mathbb H)):=&\mathcal L^2(\Omega,\mathcal F,\mathbf P;L(U,\mathbb H))\\
=&\bigg\{Y:\Omega\to L(U,\mathbb H)\bigg|\mathbb E\|Y\|_{L(U,\mathbb H)}^2=\int_\Omega\|Y\|^2_{L(U,\mathbb H)}{\rm d}\mathbf P<\infty\bigg\}
\end{align*}
and define a norm by
\[
\|Y\|_{\mathcal L^2(\mathbf P;L(U,\mathbb H))}:=\left(\int_\Omega\|Y\|^2_{L(U,\mathbb H)} {\rm d}\mathbf P\right)^\frac{1}{2}.
\]
Note that $(L(U,\mathbb{H}),\|\cdot\|_{L(U,\mathbb{H})})$ and
$(\mathcal L^2(\mathbf P;L(U,\mathbb H)),\|\cdot\|_{\mathcal L^2(\mathbf P;L(U,\mathbb H))})$
are Banach spaces. Define
\begin{align*}
\mathcal L^2(\mathbf P_\nu;\mathbb H)&:= \mathcal L^2(\Omega\times U,\mathcal P_U,\mathbf P_\nu; \mathbb H)  \\
&=\bigg\{Y:\Omega\times U\to \mathbb H \bigg|
\int_{\Omega\times U}|Y|^2{\rm d }\mathbf P_\nu=\int_U\mathbb E|Y|^2\nu({\rm d}x)<\infty \bigg\},
\end{align*}
where $\mathcal P_U$ is the product $\sigma$-algebra on $\Omega\times U$ and $\mathbf P_\nu=\mathbf P\otimes \nu$.
For $Y\in \mathcal L^2(\mathbf P_\nu;\mathbb H)$, let
\[
\|Y\|_{\mathcal L^2(\mathbf P_\nu;\mathbb H)}:=\left(\int_U\mathbb E|Y|^2\nu({\rm d }x)\right)^\frac{1}{2}.
\]
Then $\mathcal L^2(\mathbf P_\nu;\mathbb H)$ is a Hilbert space equipped with the norm $\|\cdot\|_{\mathcal L^2(\mathbf P_\nu;\mathbb H)}$.

\begin{remark}\label{gQ} \rm
If the operator $\mathcal Q\in L(U)$, the space of bounded linear operators from $U$ to $U$, is nonnegative,
symmetric and ${\rm Tr}\mathcal Q<\infty$, then $L\mathcal Q^\frac{1}{2}\in L_2(U,\mathbb H)$ for all $L\in L(U,\mathbb H)$,
where the space $L_2(U,\mathbb H)$ is a separable Hilbert space that consists of all Hilbert-Schmidt operators
from $U$ to $\mathbb H$ with inner product $\langle A,B\rangle_{L_2(U,\mathbb H)}:=\sum\limits_{k\in \mathbb N}\langle Ae_k,Be_k\rangle$ and
$\{e_k\}_{k\in \mathbb N}$ an orthonormal basis of $U$.
For $g\in \mathcal L^2(\mathbf P;L(U,\mathbb H))$, we have $g\mathcal Q^\frac{1}{2}\in\mathcal L^2(\mathbf P;L_2(U,\mathbb H))$
and denote $\|g\mathcal Q^\frac{1}{2}\|_{\mathcal L^2(\mathbf P;L_2(U,\mathbb H))}
=\left(\mathbb E\|g\mathcal Q^\frac{1}{2}\|^2_{L_2(U,\mathbb H)}\right)^\frac{1}{2}$.
\end{remark}

In this section we consider the following linear SDE perturbed by L\'evy noise
\begin{align}\label{SlSDE}
{\rm d }Y(t)=&(AY(t)+f(t)){\rm d }t+g(t){\rm d }W(t) \\
&+\int_{|x|_U<1}F(t,x)\widetilde{N}({\rm d }t,{\rm d }x)
+\int_{|x|_U\geq1}G(t,x)N({\rm d }t,{\rm d }x),\qquad t\in\mathbb{R},\nonumber
\end{align}
where $A$ is an infinitesimal generator which generates a $\mathcal C^0$-semigroup $\{T(t)\}_{t\geq0}$ on
$\mathbb{H}$ and $f:\mathbb{R}\to\mathcal{L}^{2}(\mathbf{P};\mathbb{H})$,
$g:\mathbb R \to \mathcal L^2(\mathbf P;L(U,\mathbb H))$,
$F, G:\mathbb R \to\mathcal L^2(\mathbf P_\nu;\mathbb H)$. Here $W$ and $N$ are the L\'evy-It\^o decomposition
components of the two-sided L\'evy process {\em L} as in Section \ref{levy}.

\begin{theorem}\label{thmone}
Consider {\rm (\ref{SlSDE})}. Assume that $A$ generates a dissipative $\mathcal C^0$-semigroup $\{T(t)\}_{t\geq0}$ such that
\begin{equation}\label{exp}
\left\|T(t)\right\|\leq K {\rm e}^{-\omega t} \qquad \mbox {for all}~~~  t\geq 0,
\end{equation}
with $K>0$, $\omega>0$; $f\in C_b(\mathbb{R},\mathcal{L}^{2}(\mathbf{P};\mathbb{H}))$,
$g\in C_b(\mathbb R,\mathcal L^2(\mathbf P;L(U,\mathbb H)))$, and
$F, G\in C_b(\mathbb R,\mathcal L^2(\mathbf P_\nu;\mathbb H))$. Suppose that
$W$ and $N$ are the same as in Section \ref{levy}.
Then {\rm (\ref{SlSDE})} has a unique mild solution
$\varphi\in C_b(\mathbb{R},\mathcal{L}^{2}(\mathbf{P};\mathbb{H}))$ and this unique $\mathcal L^2$-bounded solution
is strongly compatible in distribution. Furthermore, $\mathfrak M_{(f,g,F,G)}^{u}\subseteq \mathfrak {\tilde{M}}_{\varphi}^{u}$,
where $\mathfrak{\tilde{M}}_{\varphi}^{u}$ is the set of all sequences $\{t_n\}$ such that the sequence
$\{\varphi(t+t_n)\}$ converges in distribution uniformly in $t\in \mathbb{R}$.
\end{theorem}

\begin{proof}
{\bf Step 1. There exists a unique mild solution $\varphi\in C_b(\mathbb{R},\mathcal{L}^{2}(\mathbf{P};\mathbb{H}))$.}
By (\ref{SDEA}) and {\rm(\ref{exp})}, if $\varphi$ is $\mathcal L^2$-bounded,
then it is a mild solution of {\rm(\ref{SlSDE})} if and only if it is given by the following form
\begin{align}\label{SDEA1}
\varphi(t)=&\int_{-\infty}^t T(t-\tau)f(\tau){\rm d }\tau+\int_{-\infty}^t T(t-\tau)g(\tau){\rm d }W(\tau)  \\
 &+\int_{-\infty}^t\int_{|x|_U<1}T(t-\tau)F(\tau,x)\widetilde{N}({\rm d }\tau,{\rm d }x)
 +\int_{-\infty}^t\int_{|x|_U\geq1}T(t-\tau)G(\tau,x)N({\rm d }\tau,{\rm d }x)\nonumber.
\end{align}
First we prove the uniqueness of the $\mathcal L^2$-bounded solution. If $\psi\in C_b(\mathbb{R},\mathcal{L}^{2}(\mathbf{P};\mathbb{H}))$
is another solution of ({\ref{SlSDE}}), then $\varphi(t)-\psi(t)$ is the solution of the equation ${\rm d }Y(t)=AY(t){\rm d }t$.
Since {\rm(\ref{exp})} holds, this equation has only zero solution in
$C_b(\mathbb{R},\mathcal{L}^{2}(\mathbf{P};\mathbb{H}))$. Therefore $\varphi=\psi$.

Now we prove that $\varphi$ given by \eqref{SDEA1} is $\mathcal L^2$-bounded, i.e.
$\sup\limits_{t\in\mathbb R}\|\varphi(t)\|_{\mathcal{L}^{2}(\mathbf{P};\mathbb{H})}<\infty$. According to (\ref{SDEA1}),
\begin{align}\label{3.5}
\mathbb E|\varphi(t)|^2 \leq& 4\mathbb{E}\left|\int_{-\infty}^t T(t-\tau)f(\tau){\rm d }\tau \right|^2
+4\mathbb{E}\left|\int_{-\infty}^{t}T(t-\tau)g(\tau){\rm d }W(\tau)\right|^2  \\
 &+4\mathbb{E}\left|\int_{-\infty}^{t}\int_{|x|_U<1}T(t-\tau)F(\tau,x)\widetilde{N}({\rm d }\tau,{\rm d }x)\right|^2 \nonumber \\
 &+4\mathbb{E}\left|\int_{-\infty}^{t}\int_{|x|_U\geq1}T(t-\tau)G(\tau,x)N({\rm d }\tau,{\rm d }x)\right|^2. \nonumber
\end{align}
For the first term of (\ref{3.5}), by Cauchy-Schwarz inequality we have
\begin{align}\label{third seven}
\mathbb{E}\left|\int_{-\infty}^t T(t-\tau)f(\tau){\rm d }\tau \right|^2
&\leq \mathbb E\left(\int_{-\infty}^{t}Ke^{-\omega(t-\tau)}|f(\tau)|{\rm d }\tau \right)^2  \\
&\leq K^2\int_{-\infty}^{t}{\rm e}^{-\omega(t-\tau)}{\rm d }\tau\int_{-\infty}^{t}{\rm e}^{-\omega(t-\tau)}\mathbb E|f(\tau)|^2{\rm d }\tau \nonumber\\
&\leq\frac{K^2}{\omega^2}\sup_{\tau\in\mathbb R}\mathbb E|f(\tau)|^2. \nonumber
\end{align}
For the second term, using It\^o's isometry property we get
\begin{align}\label{third eight}
\mathbb{E}\left|\int_{-\infty}^{t}T(t-\tau)g(\tau){\rm d }W(\tau)\right|^2
&=\int_{-\infty}^{t}\mathbb{E}\|T(t-\tau)g(\tau)\mathcal Q^\frac{1}{2}\|^2_{L_2(U,\mathbb{H})}{\rm d }\tau \\
&\leq\int_{-\infty}^{t}K^2{\rm e}^{-2\omega(t-\tau)}
\mathbb{E}\|g(\tau)\mathcal{Q}^\frac{1}{2}\|_{L_2(U,\mathbb{H})}^2{\rm d }\tau \nonumber\\
&\leq\frac{K^2}{2\omega}\sup_{\tau\in\mathbb R}\mathbb{E}\|g(\tau)\mathcal{Q}^\frac{1}{2}\|_{L_2(U,\mathbb{H})}^2. \nonumber
\end{align}
For the third term, by properties of integrals for Poisson random measures we gain
\begin{align}\label{third nine}
\mathbb{E}\left|\int_{-\infty}^{t}\int_{|x|_U<1}T(t-\tau)F(\tau,x)\widetilde{N}
({\rm d }\tau,{\rm d }x)\right|^2
&=\int_{-\infty}^{t}\int_{|x|_U<1}\mathbb{E}|T(t-\tau)F(\tau,x)|^2\nu({\rm d }x){\rm d }\tau  \\
&\leq\int_{-\infty}^{t}\int_{|x|_U<1}K^2{\rm e}^{-2\omega(t-\tau)}\mathbb{E}|F(\tau,x)|^2\nu({\rm d }x){\rm d }\tau \nonumber\\
&\leq\frac{K^2}{2\omega}\sup_{\tau\in\mathbb R}\int_{|x|_U<1}\mathbb{E}|F(\tau,x)|^2\nu({\rm d }x). \nonumber
\end{align}
For the fourth term, by properties of integrals for Poisson random measures and Cauchy-Schwarz inequality, we obtain
\begin{align}\label{third ten}
\mathbb{E}&\left|\int_{-\infty}^{t}\int_{|x|_U\geq1}T(t-\tau)G(\tau,x)N({\rm d }\tau,{\rm d }x)\right|^2 \\
&\leq2\mathbb{E}\left|\int_{-\infty}^{t}\int_{|x|_U\geq1}
T(t-\tau)G(\tau,x)\widetilde{N}({\rm d }\tau,{\rm d }x)\right|^2 +2\mathbb{E}\left|\int_{-\infty}^{t}\int_{|x|_U\geq1}
T(t-\tau)G(\tau,x)\nu({\rm d }x){\rm d }\tau\right|^2  \nonumber\\
&\leq2\int_{-\infty}^{t}\int_{|x|_U\geq1}K^2{\rm e}^{-2\omega(t-\tau)}\mathbb{E}|G(\tau,x)|^2\nu({\rm d }x){\rm d }\tau\nonumber\\
&\quad+2\int_{-\infty}^{t}\int_{|x|_U\geq1}K^2{\rm e}^{-\omega(t-\tau)}
\nu({\rm d }x){\rm d }\tau
\cdot \int_{-\infty}^{t}\int_{|x|_U\geq1}{\rm e}^{-\omega(t-\tau)}\mathbb{E}|G(\tau,x)|^2\nu({\rm d }x){\rm d }\tau \nonumber \\
&\leq\left(\frac{K^2}{\omega}+\frac{2K^2b}{\omega^2}\right)\sup_{\tau\in\mathbb R}\int_{|x|_U\geq1}\mathbb{E}|G(\tau,x)|^2\nu({\rm d }x). \nonumber
\end{align}
From (\ref{3.5})-(\ref{third ten}), we have for any $t\in\mathbb R$
\begin{align}\label{item2}
\mathbb{E}|\varphi(t)|^2\leq&\frac{K^2}{\omega^2}\bigg(4\sup_{\tau\in\mathbb R}\mathbb E|f(\tau)|^2+{2\omega}\sup_{\tau\in\mathbb R}\mathbb{E}\|g(\tau)\mathcal{Q}^\frac{1}{2}\|_{L_2(U,\mathbb{H})}^2
+{2\omega}\sup_{\tau\in\mathbb R}\int_{|x|_U<1}\mathbb{E}|F(\tau,x)|^2\nu({\rm d }x)\\
&\quad+(4\omega+8b)\sup_{\tau\in\mathbb R}\int_{|x|_U\geq1}\mathbb{E}|G(\tau,x)|^2\nu({\rm d }x)\bigg),\nonumber
\end{align}
and hence
\begin{equation}\label{item22}
\|\varphi\|_{\infty}^2\leq\frac{K^2}{\omega^2}\left(4\|f\|_{\infty}^2+2\omega\|g\mathcal Q^\frac{1}{2}\|_{\infty}^2
+2\omega\|F\|_\infty^2+(4\omega+8b)\|G\|_\infty^2\right).
\end{equation}

Next we prove that $\varphi$ is $\mathcal L^2$-continuous. To verify this, we can weaken the
condition of $\{T(t)\}_{t\geq 0}$ to a $\mathcal C^0$-semigroup, not necessarily dissipative.
In other words, there exist constants $M>0$ and $\delta>0$ such that $\|T(t)\|\le M \rm e^{\delta t}$ holds for all $t\ge 0$.
Since $\varphi$ is a $\mathcal L^2$-bounded solution of (\ref{SlSDE}), i.e. it satisfies (\ref{SDEA}),
then similar to (\ref{third seven})-(\ref{third ten}) it follows from Cauchy-Schwarz inequality,
It\^o's isometry property and properties of integrals for Poisson random measures that for $t\geq r$
\begin{align*}
\mathbb E &|\varphi(t)-\varphi(r)|^2 \\
&\leq 5\mathbb E|T(t-r)\varphi(r)-\varphi(r)|^2
+5\mathbb E\left|\int_r^tT(t-\tau)f(\tau){\rm d }\tau\right|^2 \\
&\quad+5\mathbb E\left|\int_r^tT(t-\tau)g(\tau){\rm d }W(\tau)\right|^2
+5\mathbb E\left|\int_r^t\int_{|x|_U<1}T(t-\tau)F(\tau,x)\widetilde N({\rm d }\tau,{\rm d }x)\right|^2 \\
&\quad+5\mathbb E\left|\int_r^t\int_{|x|_U\geq1}T(t-\tau)G(\tau,x)\widetilde N({\rm d }\tau,{\rm d }x)
+\int_r^t\int_{|x|_U\geq1}T(t-\tau)G(\tau,x)\nu({\rm d }x){\rm d }\tau\right|^2 \\
&\leq 5\mathbb E|T(t-r)\varphi(r)-\varphi(r)|^2
+5 M^2 {\rm e}^{2\delta(t-r)} \int_r^t 1 {\rm d }\tau \cdot \int_r^t \mathbb{E} |f(\tau)|^2 {\rm d }\tau \\
&\quad+5 M^2 {\rm e}^{2\delta(t-r)} \int_r^t \mathbb{E} \|g(\tau)\mathcal Q^\frac{1}{2}\|^2_{L_2(U,\mathbb H)} {\rm d }\tau
+5 M^2 {\rm e}^{2\delta(t-r)} \int_{r}^{t} \int_{|x|_U<1} \mathbb E |F(\tau,x)|^2\nu({\rm d }x){\rm d }\tau \\
&\quad+5 M^2 {\rm e}^{2\delta(t-r)}\bigg[2 \int_{r}^{t} \int_{|x|_U\geq
1} \mathbb{E} |G(\tau,x)|^2 \nu({\rm d }x){\rm d }\tau  \\
&\qquad+ 2\int_{r}^{t}\int_{|x|_U\geq 1} 1 \nu({\rm d }x){\rm d }\tau \cdot \int_{r}^{t} \int_{|x|_U\geq 1}
 \mathbb{E} |G(\tau,x)|^2 \nu({\rm d }x){\rm d }\tau \bigg].
\end{align*}
Note that $\{T(t)\}_{t\geq 0}$ is a $\mathcal C^0$-semigroup, so $|T(t-r)x-x|\to 0$ for any $x\in\mathbb H$ as $t\to r^+$.
Since $\|T(t)\|\leq M \rm e^{\delta t}$, $t\geq 0$, we have
\begin{align*}
\mathbb E|T(t-r)\varphi(r)-\varphi(r)|^2&=\mathbb E|(T(t-r)-Id)\varphi(r)|^2 \\
&\leq2(\|T(t-r)\|^2+1)\mathbb E|\varphi(r)|^2< \infty.
\end{align*}
It follows from Lebesgue dominated convergence theorem that $\mathbb{E}|T(t-r)Y(r)-Y(r)|^2\to 0$, as $t\to r^+$.
For $f\in C_b(\mathbb{R},\mathcal{L}^{2}(\mathbf{P};\mathbb{H}))$, we have
\begin{align*}
\sup_{\tau\in\mathbb R}\mathbb E|f(\tau)|^2<\infty;
\end{align*}
and analogically
\begin{align*}
\sup_{\tau\in\mathbb R}\mathbb{E} \|g(\tau)\mathcal Q^\frac{1}{2}\|^2_{L_2(U,\mathbb H)} <\infty,
\end{align*}

\begin{align*}
\sup_{\tau\in\mathbb R}\int_{|x|_U<1}\mathbb E|F(\tau,x)|^2\nu({\rm d }x)<\infty,
\end{align*}

\begin{align*}
\sup_{\tau\in\mathbb R}\int_{|x|_U\geq 1}\mathbb E|G(\tau,x)|^2\nu({\rm d }x)<\infty.
\end{align*}
Hence $\mathbb E|\varphi(t)-\varphi(r)|^2\to 0$ as $t\to r^{+}$. Using same arguments, we get
$\mathbb E|\varphi(t)-\varphi(r)|^2\to 0$ as $t\to r^{-}$. So $\varphi$ is $\mathcal L^2$-continuous.

{\bf Step 2. The unique $\mathcal L^2$-bounded solution is strongly compatible in distribution.}
Let $l>k>0$ and $t\in[-k,k]$. Note that
\begin{align}\label{11}
&\max_{|t|\leq k}\int_{-\infty}^{t}{\rm e}^{-\omega(t-\tau)}\mathbb{E}|f(\tau)|^2{\rm d }\tau\\
&\quad\leq\max_{|t|\leq k}\int_{-l}^{t}{\rm e}^{-\omega(t-\tau)}\mathbb{E}|f(\tau)|^2{\rm d }\tau
+\max_{|t|\leq k}\int_{-\infty}^{-l}{\rm e}^{-\omega(t-\tau)}\mathbb{E}|f(\tau)|^2{\rm d }\tau\nonumber \\
&\quad\leq\frac{1}{\omega}\max_{|\tau|\leq l}\mathbb{E}|f(\tau)|^2+\frac{1}{\omega}{\rm e}^{-\omega(l-k)}\sup_{\tau\in\mathbb R}\mathbb E|f(\tau)|^2. \nonumber
\end{align}
Then  by (\ref{11}), from (\ref{third seven})-(\ref{third ten}) we respectively have
\begin{align}\label{third eleven}
\max_{|t|\leq k}\mathbb{E}\left|\int_{-\infty}^tT(t-\tau)f(\tau){\rm d }\tau\right|^2
&\leq\max_{|t|\leq k}\frac{K^2}{\omega}\int_{-\infty}^{t}{\rm e}^{-\omega(t-\tau)}\mathbb{E}|f(\tau)|^2{\rm d }\tau  \\
&\leq\frac{K^2}{\omega^2}\max_{|\tau|\leq l}\mathbb{E}|f(\tau)|^2+\frac{K^2}{\omega^2}{\rm e}^{-\omega(l-k)}\|f\|_\infty^2,\nonumber
\end{align}
\begin{align}\label{third twelve}
&\max_{|t|\leq k}\mathbb{E}\left|\int_{-\infty}^tT(t-\tau)g(\tau){\rm d }W(\tau)\right|^2 \\
&\quad\leq\max_{|t|\leq k}K^2\int_{-\infty}^{t}{\rm e}^{-2\omega(t-\tau)}
\mathbb{E}\|g(\tau)\mathcal{Q}^\frac{1}{2}\|_{L_2(U,\mathbb{H})}^2{\rm d }\tau \nonumber \\
&\quad\leq\frac{K^2}{2\omega}\max_{|\tau|\leq l}\mathbb{E}\|g(\tau)\mathcal{Q}^\frac{1}{2}\|_{L_2(U,\mathbb{H})}^2
+\frac{K^2}{2\omega}{\rm e}^{-2\omega(l-k)}\|g\mathcal Q^\frac{1}{2}\|_\infty^2,\nonumber
\end{align}
\begin{align}\label{third thirteen}
&\max_{|t|\leq k}\mathbb E\left|\int_{-\infty}^{t}\int_{|x|_U<1}T(t-\tau)F(\tau,x)\widetilde N({\rm d }\tau,{\rm d }x)\right|^2 \\
&\quad\leq\max_{|t|\leq k}K^2\int_{-\infty}^{t}\int_{|x|_U<1}
{\rm e}^{-2\omega(t-\tau)}\mathbb{E}|F(\tau,x)|^2\nu({\rm d }x){\rm d }\tau  \nonumber\\
&\quad\leq\frac{K^2}{2\omega}\max_{|\tau|\leq l}\int_{|x|_U<1}\mathbb E|F(\tau,x)|^2\nu({\rm d }x)
+\frac{K^2}{2\omega}{\rm e}^{-2\omega(l-k)}\|F\|_\infty^2,\nonumber
\end{align}
and
\begin{align}\label{third fourteen}
&\max_{|t|\leq k}\mathbb E\left|\int_{-\infty}^{t}\int_{|x|_U\geq1}T(t-\tau)G(\tau,x)N({\rm d }\tau,{\rm d }x)\right|^2\\
&\leq\max_{|t|\leq k}2K^2\int_{-\infty}^{t}\int_{|x|_U\geq1}{\rm e}^{-2\omega(t-\tau)}\mathbb{E}|G(\tau,x)|^2\nu({\rm d }x){\rm d }\tau \nonumber\\
&\quad+\max_{|t|\leq k}2K^2\int_{-\infty}^{t}\int_{|x|_U\geq1}{\rm e}^{-\omega(t-\tau)}\nu({\rm d }x){\rm d }\tau
\cdot\max_{|t|\leq k}\int_{-\infty}^{t}\int_{|x|_U\geq1}{\rm e}^{-\omega(t-\tau)}
\mathbb{E}|G(\tau,x)|^2\nu({\rm d }x){\rm d }\tau \nonumber\\
&\leq\left(\frac{K^2}{\omega}+\frac{2K^2b}{\omega^2}\right)
\max_{|\tau|\leq l}\int_{|x|_U\geq1}\mathbb E|G(\tau,x)|^2\nu({\rm d }x)
+\left(\frac{K^2}{\omega}{\rm e}^{-2\omega(l-k)}
+\frac{2K^2b}{\omega^2}{\rm e}^{-\omega(l-k)}\right)\|G\|_\infty^2.\nonumber
\end{align}
According to (\ref{3.5}) and (\ref{third eleven})-(\ref{third fourteen}), we obtain
\begin{align}\label{itemthird}
&\max_{|t|\leq k}\mathbb E|\varphi(t)|^2 \\
&\quad\leq\frac{2K^2}{\omega ^2}
\bigg{(}2\max\limits_{|\tau|\leq l}\mathbb E{|f(\tau)|^2}
+\omega\max\limits_{|\tau|\leq l} \mathbb E{\|g(\tau)\mathcal Q^\frac{1}{2}\|^2_ {L_2(U,\mathbb{H})}} \nonumber\\
&\quad\qquad+\omega\max_{|\tau|\leq l}\int_{|x|_U<1}\mathbb E|F(\tau,x)|^2\nu({\rm d }x)
+(2\omega+4b)\max_{|\tau|\leq l}\int_{|x|_U\geq1}\mathbb E|G(\tau,x)|^2\nu({\rm d }x)\bigg) \nonumber\\
&\qquad+\frac{2K^2}{\omega^2}\bigg(2{\rm e}^{-\omega(l-k)}\|f\|_\infty^2+\omega {\rm e}^{-2\omega(l-k)}
\|g\mathcal Q^\frac{1}{2}\|_\infty^2
+\omega {\rm e}^{-2\omega(l-k)}\|F\|_\infty^2 \nonumber \\
&\quad\qquad+\left(2\omega {\rm e}^{-2\omega(l-k)}+4b{\rm e}^{-\omega(l-k)}\right)\|G\|_\infty^2\bigg). \nonumber
\end{align}

Let $\{t_n\}\in\mathfrak M_{(f,g,F,G)}$.
Then there exists $(\widetilde f,\widetilde g,\widetilde F,\widetilde G)\in H(f,g,F,G)$ such that for given $k>0$
\begin{equation}\label{third sixteen}
\lim_{n\rightarrow \infty}\max_{|t|\leq k}\mathbb{E}|f(t+t_n)-\widetilde f(t)|^2=0,
\end{equation}
\begin{equation}\label{third sixteen2}
\lim_{n\rightarrow \infty}\max_{|t|\leq k}\mathbb{E}\|(g(t+t_n)-\widetilde g(t))\mathcal Q^\frac{1}{2}\|^2_{L_2(U,\mathbb{H})}=0,
\end{equation}
\begin{equation}\label{third sixteen3}
\lim_{n\rightarrow \infty}\max_{|t|\leq k}\int_{|x|_U<1}\mathbb{E}|F(t+t_n,x)-\widetilde F(t,x)|^2\nu({\rm d }x)=0,
\end{equation}
\begin{equation}\label{third sixteen4}
\lim_{n\rightarrow \infty}\max_{|t|\leq k}\int_{|x|_U\geq1}\mathbb{E}|G(t+t_n,x)-\widetilde G(t,x)|^2\nu({\rm d }x)=0.
\end{equation}
Let $\varphi_n$ denote the unique $\mathcal {L}^2$-bounded solution of the equation
\begin{align*}
{\rm d }Y(t)=&(AY(t)+f^{t_n}(t)){\rm d }t+g^{t_n}(t){\rm d }W(t)+\int_{|x|_U<1}F^{t_n}(t,x)\widetilde N({\rm d }t,{\rm d }x) \\
&+\int_{|x|_U\geq1}G^{t_n}(t,x)N({\rm d }t,{\rm d }x) \qquad  (n\in \mathbb{N}, t\in\mathbb{R}).
\end{align*}
Let $\widetilde\varphi$ denote the unique $\mathcal {L}^2$-bounded solution of the equation
\begin{align*}
{\rm d }Y(t)=&(AY(t)+\widetilde f(t)){\rm d }t+\widetilde g(t){\rm d }W(t)+\int_{|x|_U<1}\widetilde F(t,x)\widetilde N({\rm d }t,{\rm d }x)   \\
&+\int_{|x|_U\geq1}\widetilde G(t,x)N({\rm d }t,{\rm d }x)\qquad  (t\in\mathbb{R}).
\end{align*}
Then $\psi_n:= \varphi_n-\widetilde\varphi$ is the solution of the following equation
\begin{align*}
{\rm d }Y(t)=&(AY(t)+(f^{t_n}(t)-\widetilde f(t))){\rm d }t+(g^{t_n}(t)-\widetilde g(t)){\rm d }W(t) \\
&+\int_{|x|_U<1}(F^{t_n}(t,x)-\widetilde F(t,x))\widetilde N({\rm d }t,{\rm d }x)\\
&+\int_{|x|_U\geq1}(G^{t_n}(t,x)-\widetilde G(t,x))N({\rm d }t,{\rm d }x) \qquad  (n\in \mathbb{N}, t\in\mathbb{R}),
\end{align*}
where $f^{t_n}-\widetilde f\in C_b(\mathbb R,\mathcal{L}^2(\mathbf{P};\mathbb{H}))$,
$g^{t_n}-\widetilde g\in C_b(\mathbb R,\mathcal L^2(\mathbf P;L(U,\mathbb H)))$,
$F^{t_n}-\widetilde F, G^{t_n}-\widetilde G\in  C_b(\mathbb R,\mathcal L^2(\mathbf P_\nu;\mathbb H))$.
By Remark \ref{rem hull}, we have
\begin{equation*}
\|f^{t_n}-\widetilde f\|_\infty\leq 2\|f\|_\infty,\quad \|(g^{t_n}-\widetilde g)\mathcal Q^{\frac{1}{2}}\|_\infty
\leq 2\|g\mathcal Q^{\frac{1}{2}}\|_\infty,
\end{equation*}
\begin{equation*}
\|F^{t_n}-\widetilde F\|_\infty\leq 2\|F\|_\infty, \quad \|G^{t_n}-\widetilde G\|_\infty\leq 2\|G\|_\infty.
\end{equation*}
Now let ${l_n}$ be a sequence of positive numbers such that ${l_n}\to\infty$ as $n\to\infty$. On the basis of (\ref{itemthird}), we get
\begin{align}\label{third seventeen}
&\max_{|t|\leq k}\mathbb E|\psi_n(t)|^2 \\
&\quad\leq \frac{2K^2}{\omega^2}\bigg{(}2\max_{|\tau|\leq l_n}\mathbb E|f^{t_n}(\tau)-\widetilde f(\tau)|^2
+\omega\max_{|\tau|\leq l_n}\mathbb E\|(g^{t_n}(\tau)-\widetilde g(\tau))\mathcal Q^\frac{1}{2}\|^2_{L_2(U,\mathbb{H})} \nonumber\\
&\quad\qquad +\omega\max_{|\tau|\leq l_n}\int_{|x|_U<1}\mathbb E|F^{t_n}(\tau,x)-\widetilde F(\tau,x)|^2\nu({\rm d }x)\nonumber\\
&\quad\qquad +(2\omega+4b)\max_{|\tau|\leq l_n}\int_{|x|_U\geq1}\mathbb E|G^{t_n}(\tau,x)-\widetilde G(\tau,x)|^2
\nu({\rm d }x)\bigg{)} \nonumber \\
&\qquad+\frac{2K^2}{\omega^2}\bigg{(}2{\rm e}^{-\omega(l_n-k)}\|f^{t_n}-\widetilde f\|_\infty^2
+\omega {\rm e}^{-2\omega(l_n-k)}\|(g^{t_n}-\widetilde g)\mathcal Q^{\frac{1}{2}}\|_\infty^2   \nonumber\\
&\quad\qquad+\omega {\rm e}^{-2\omega(l_n-k)}\|F^{t_n}-\widetilde F\|_\infty^2+\left(2\omega {\rm e}^{-2\omega(l_n-k)}
+4b{\rm e}^{-\omega(l_n-k)}\right)\|G^{t_n}-\widetilde G\|_\infty^2\bigg{)}.\nonumber
\end{align}
Letting $n\to\infty$ in (\ref{third seventeen}) and considering Remark \ref{remD1}-2-(iii), by \eqref{third sixteen}-\eqref{third sixteen4}
we have
\begin{align*}
\lim_{n\to\infty}\max_{|t|\leq k}\mathbb E|\psi_n(t)|^2=0
\end{align*}
for any $k>0$; that is, $\varphi_n\to\widetilde\varphi$ in $C(\mathbb R,\mathcal{L}^2(\mathbf{P};\mathbb{H}))$ as $n\to\infty$.
Since $\mathcal{L}^2$ convergence implies convergence in distribution, we have $\varphi_n\to\widetilde\varphi$
in distribution uniformly in $t\in[-k,k]$ for all $k>0$. Let $\tau=\sigma+t_n$ in (\ref{SDEA1}). Then
\begin{align*}
\varphi(t+t_n)=&\int_{-\infty}^{t}T(t-\sigma)f(\sigma+t_n){\rm d }\sigma+\int_{-\infty}^{t}T(t-\sigma)g(\sigma+t_n){\rm d }W_n(\sigma) \\
&+\int_{-\infty}^{t}\int_{|x|_U<1}T(t-\sigma)F(\sigma+t_n,x)\widetilde{N}_n({\rm d }\sigma,{\rm d}x)\\
&+\int_{-\infty}^{t}\int_{|x|_U\geq1}T(t-\sigma)G(\tau+t_n,x)N_n({\rm d }\sigma,{\rm d}x)
\end{align*}
where $W_n(\sigma):=W(\sigma+t_n)-W(t_n)$, $N_n(\sigma,x):=N(\sigma+t_n,x)-N(t_n,x)$ and
$\widetilde {N}_n(\sigma,x):=\widetilde {N}(\sigma+t_n,x)-\widetilde {N}(t_n,x)$ for each
$\sigma\in\mathbb R$. Since $W_n$ is a $\mathcal Q$-Wiener process with the same law as
$W$ and $N_n$ has the same law as $N$ with compensated Poisson random measure $\widetilde{N}_n$,
$\varphi_n(t)$ and $\varphi(t+t_n)$ share the same distribution on $\mathbb H$.
Hence $\varphi(t+t_n)\to\widetilde\varphi(t)$ in distribution uniformly in $t\in[-k,k]$ for all $k>0$.
Thus we have $\{t_n\}\in \mathfrak{\tilde{M}}_{\varphi}$. That is, $\varphi$ is strongly compatible in distribution.

{\bf Step 3. $\mathfrak M_{(f,g,F,G)}^{u}\subseteq \mathfrak {\tilde{M}}_{\varphi}^{u}$.}
Let$\{t_n\}\in\mathfrak{M}_{(f,g,F,G)}^{u}$. Then there exists $(\widetilde f,\widetilde g,\widetilde F,\widetilde G)\in H(f,g,F,G)$ such that
\begin{equation}\label{third eighteen}
\lim_{n\rightarrow \infty}\sup_{t\in\mathbb R}\mathbb{E}|f(t+t_n)-\widetilde f(t)|^2=0,
\end{equation}
\begin{equation}\label{third eighteen2}
\lim_{n\rightarrow \infty}\sup_{t\in\mathbb R}\mathbb{E}\|(g(t+t_n)-\widetilde g(t))\mathcal Q^\frac{1}{2}\|^2_{L_2(U,\mathbb{H})}=0,
\end{equation}
\begin{equation}\label{third eighteen3}
\lim_{n\rightarrow \infty}\sup_{t\in\mathbb R}\int_{|x|_U<1}\mathbb{E}|F(t+t_n,x)-\widetilde F(t,x)|^2\nu({\rm d }x)=0,
\end{equation}
\begin{equation}\label{third eighteen4}
\lim_{n\rightarrow \infty}\sup_{t\in\mathbb R}\int_{|x|_U\geq1}\mathbb{E}|G(t+t_n,x)-\widetilde G(t,x)|^2\nu({\rm d }x)=0.
\end{equation}
We define $\varphi_n$, $\widetilde\varphi$ and $\psi_n$ as in Step 2. According to (\ref{item22}), we obtain
\begin{align}\label{third nineteen}
\|\psi_n\|_\infty^2 \leq&\frac{K^2}{\omega^2}\Big(4\|f^{t_n}-\widetilde f\|_{\infty}^2
+2\omega\|(g^{t_n}-\widetilde g)\mathcal Q^\frac{1}{2}\|_{\infty}^2 \\
&\quad+2\omega\|F^{t_n}-\widetilde F\|_\infty^2+(4\omega+8b)\|G^{t_n}-\widetilde G\|_\infty^2\Big). \nonumber
\end{align}
Passing to limit in (\ref{third nineteen}), by \eqref{third eighteen}-\eqref{third eighteen4} we get
$\varphi_n\to\widetilde\varphi $ uniformly on $\mathbb R$ in $\mathcal L^2$-norm as $n\to\infty$.
Since $\varphi_n(t)$ and $\varphi(t+t_n)$ have the same distributions, $\varphi(t+t_n)\to\widetilde\varphi(t)$
in distribution uniformly in $t\in\mathbb R$. Thus we have $\{t_n\}\in \mathfrak {\tilde{M}}_{\varphi}^{u}$.

The proof is complete.
\end{proof}

\begin{coro}\label{co third one}

Consider \eqref{SlSDE}. Under the assumptions of \textsl{Theorem} {\rm \ref{thmone}}, it follows from
\textsl{Theorems} {\rm \ref{th1}} and {\rm\ref{thmone}} that
\begin{enumerate}
\item
if coefficients $f,g,F,G$ are jointly stationary (respectively,
$\tau$-periodic, quasi-periodic with the spectrum of frequencies $\nu_1,\ldots,\nu_m$,
almost periodic, almost automorphic, Birkhoff recurrent, Lagrange stable,
Levitan almost periodic, almost recurrent, Poisson stable), then the unique $\mathcal L^2$-bounded solution of {\rm (\ref{SlSDE})} has the same
recurrent property as coefficients in distribution;

\item
if coefficients $f,g,F,G$ are jointly Lagrange stable and jointly pseudo-periodic (or
pseudo-recurrent), then the unique $\mathcal L^2$-bounded solution of {\rm (\ref{SlSDE})} is
pseudo-periodic (or pseudo-recurrent) in distribution.

\end{enumerate}
\end{coro}

\section{Semilinear equations}
Define
\[
\mathcal L^2( \nu;\mathbb H):= \mathcal L^2(U,\mathcal B_U, \nu; \mathbb H)=\bigg\{Y:U\to \mathbb H  \bigg| \int_U|Y|^2\nu({\rm d }x)<\infty \bigg\},
\]
with $\mathcal B_U$ being the Borel $\sigma$-algebra on $U$.
For $Y\in\mathcal L^2(\nu;\mathbb H)$, let
\[
\|Y\|_{\mathcal L^2(\nu;\mathbb H)}:=\left(\int_U|Y|^2\nu({\rm d }x)\right)^\frac{1}{2}.
\]
We use $C(\mathbb R\times\mathcal Y,\mathcal X)$ to denote the space of all continuous functions
$f:\mathbb R\times \mathcal Y\to\mathcal X$, where $(\mathcal Y, \rho)$ and $(\mathcal X, \gamma)$ are complete metric spaces.

Let us consider the semilinear SDE driven by L\'evy noise
\begin{align}\label{SemiSDE}
{\rm d }Y(t)=&(AY(t)+f(t,Y(t))){\rm d }t+g(t,Y(t)){\rm d }W(t) \\
&+\int_{|x|_U<1}F(t,Y(t-),x)\widetilde{N}({\rm d }t,{\rm d }x)
+\int_{|x|_U\geq1}G(t,Y(t-),x)N({\rm d }t,{\rm d }x), \nonumber
\end{align}
where $A$ is an infinitesimal generator of a dissipative $\mathcal C^0$-semigroup $\{T(t)\}_{t\geq0}$ on $\mathbb{H}$ such that
\begin{equation}\label{expi}
\left\|T(t)\right\|\leq K {\rm e}^{-\omega t}, \qquad \mbox {for all}~~~  t\geq 0,
\end{equation}
with $K>0$, $\omega>0$; $f\in C(\mathbb{R}\times\mathbb H,\mathbb H)$,
$g\in C(\mathbb{R}\times\mathbb H,L(U,\mathbb H))$, $F,G\in C(\mathbb{R}\times\mathbb H,\mathcal L^2(\nu;\mathbb H))$;
$W$ and $N$ are the L\'evy-It\^o decomposition components of the two-sided L\'evy process $L$ with assumptions stated
in Section \ref{levy}. We impose the following conditions on $f, g, F, G$.
\begin{enumerate}
\item[(E1)]
There exists a number $A_0\geq 0$ such that for all $t\in\mathbb R$ \\
 $|f(t,0)|\leq {A_0}$, \qquad $\|g(t,0)\mathcal Q^\frac{1}{2}\|_{L_2(U,\mathbb H)}\leq {A_0}$, \\
$\int_{|x|_U<1}|F(t,0,x)|^2\nu({\rm d }x)\leq {A_0}^2$, \qquad $\int_{|x|_U\geq1}|G(t,0,x)|^2\nu({\rm d }x)\leq {A_0}^2$.

\item[(E1$'$)]
There exists a number $A_0\geq 0$ such that for some constant $ p>2$ and all $t\in\mathbb R$ \\
 $|f(t,0)|\leq {A_0}$, \qquad $\|g(t,0)\mathcal Q^\frac{1}{2}\|_{L_2(U,\mathbb H)}\leq {A_0}$, \\
$\int_{|x|_U<1}|F(t,0,x)|^p\nu({\rm d }x)\leq {A_0}^p$, \qquad $\int_{|x|_U\geq1}|G(t,0,x)|^p\nu({\rm d }x)\leq {A_0}^p$.

\item[(E2)]
There exists a number $\mathcal L\geq 0$ such that for all $t\in\mathbb R$ and $Y_1, Y_2 \in \mathbb H$\\
$|f(t,Y_1)-f(t,Y_2)|\leq\mathcal L|Y_1-Y_2|$,\qquad $\|(g(t,Y_1)-g(t,Y_2))\mathcal Q^\frac{1}{2}\|_{L_2(U,\mathbb H)}\leq\mathcal L|Y_1-Y_2|$, \\
$\int_{|x|_U<1}|F(t,Y_1,x)-F(t,Y_2,x)|^2\nu({\rm d }x)\leq\mathcal L^2|Y_1-Y_2|^2$, \\
$\int_{|x|_U\geq1}|G(t,Y_1,x)-G(t,Y_2,x)|^2\nu({\rm d }x)\leq\mathcal L^2|Y_1-Y_2|^2$.

\item[(E2$'$)]
There exists a number $\mathcal L\geq 0$ such that for some constant $ p>2$ and $t\in\mathbb R$, $Y_1, Y_2 \in \mathbb H$\\
$|f(t,Y_1)-f(t,Y_2)|\leq\mathcal L|Y_1-Y_2|$,\qquad $\|(g(t,Y_1)-g(t,Y_2))\mathcal Q^\frac{1}{2}\|_{L_2(U,\mathbb H)}\leq\mathcal L|Y_1-Y_2|$, \\
$\int_{|x|_U<1}|F(t,Y_1,x)-F(t,Y_2,x)|^p\nu({\rm d }x)\leq\mathcal L^p|Y_1-Y_2|^p$, \\
$\int_{|x|_U\geq1}|G(t,Y_1,x)-G(t,Y_2,x)|^p\nu({\rm d }x)\leq\mathcal L^p|Y_1-Y_2|^p$.

\item[(E3)]
 $f, g, F, G$ are continuous in $t$ uniformly w.r.t. $Y$ on each bounded subset $Q\subset\mathbb H $.
\end{enumerate}

\begin{remark}\label{H(f,g,F,G)}\rm
\begin{enumerate}
\item If $f,g,F,G$ satisfy
{\rm(E1)}, {\rm(E2)}, {\rm(E3)}, then $f\in BUC(\mathbb R\times \mathbb H, \mathbb H)$,
$g\in BUC(\mathbb R\times \mathbb H, L(U,\mathbb H))$, $F,G\in BUC(\mathbb{R}\times\mathbb H,\mathcal L^2(\nu;\mathbb H))$ and
$H(f,g,F,G)\subset BUC (\mathbb R\times \mathbb H, \mathbb H) \times BUC (\mathbb R\times \mathbb H, L(U,\mathbb H))
\times BUC(\mathbb{R}\times\mathbb H,\mathcal L^2(\nu;\mathbb H))\times BUC(\mathbb{R}\times\mathbb H,\mathcal L^2(\nu;\mathbb H))$.

\item If $f,g,F,G$ satisfy {\rm(E1)} and {\rm(E2)} (or {\rm(E1$'$)} and {\rm(E2$'$)}) with the constants $A_0$
and $\mathcal L$, then every quadruplet
$(\widetilde {f},\widetilde {g},\tilde{F},\tilde{G})\in H(f,g,F,G)
:= \overline{\{(f^{\tau},g^{\tau},F^{\tau},G^{\tau}): \tau\in \mathbb R\}}$, the hull of $(f,g,F,G)$,
also possesses the same properties with the same constants.

\item  Note that {\rm(E3)} trivially holds for stochastic ordinary differential equations.
\end{enumerate}
\end{remark}

\begin{remark}\label{BDG}\rm
\begin{enumerate}
\item
{\rm(\cite[Inequality (2.24)]{K})}. Let $p>0$. Assume that $\{g(s)\}_{s\in[t_0,t]}$ is a $L(U,\mathbb H)$-valued process
on $(\Omega, \mathcal F,\mathbf P )$ with the normal filtration $\mathcal F_s$, $s\in[t_0,t]$;
$W$ is a $U$-valued $\mathcal Q$-Wiener process, where $\mathcal Q\in L(U)$ is nonnegative, symmetric and
with finite trace. Then there exists a constant $c_p>0$, depending only on $p$, such that
\begin{equation}\label{cp}
\mathbb E\sup_{s\in[t_0,t]} \left|\int_{t_0}^s g(\tau) {\rm d }W(\tau) \right|^p
\le c_p \left( \mathbb E\int_{t_0}^t \|g(\tau)\mathcal Q^{\frac{1}{2}}\|_{L_2(U,\mathbb H)}^2  {\rm d }\tau \right)^{p/2},
\end{equation}
where
\begin{equation*}
c_p:=\bigg[\frac{p(p-1)}{2}\cdot\left(\frac{p}{p-1}\right)^{p-2}\bigg]^\frac{p}{2}.
\end{equation*}

\item
By the proof of  {\cite[Theorem 2.11]{K}} and {\cite[Theorem 4.4.23]{A}} (the finite dimension case of  Kunita's first inequality),
we have the following conclusion (infinite dimension case).
Let $p\geq 2$ and $(Z, \mathcal Z, m)$ be a measurable space.
Let $N$ be a Poisson measure on $[t_0,t]\times Z$ with intensity measure $m$ and compensated measure $\widetilde N$.
Assume that $F:\Omega\times[t_0,t]\times Z\to \mathbb H$ is a predictable process
that it is $\mathcal P\times \mathcal Z$-measurable, where $\mathcal P$ is the $\sigma$-algebra on $\Omega\times[t_0,t]$
generated by left continuous adapted processes.
Then there exists a constant $d_p>0$, depending only on $p$, such that
\begin{align}\label{kunita}
\mathbb E&\sup_{s\in[t_0,t]}\left|\int_{t_0}^s\int_ZF(\tau,x)\widetilde N({\rm d }\tau,{\rm d }x)\right|^p \\
&\leq \cfrac{2^{p-3}\left(\frac{p}{p-1}\right)^p\alpha^{2-\frac{p}{2}}}{1-(p-1)(p-2)2^{p-4}\left(\frac{p}{p-1}\right)^p\alpha^{2-p}}
\mathbb E\left(\int_{t_0}^t\int_Z|F(\tau,x)|^2 m({\rm d }x){\rm d }\tau\right)^{\frac{p}{2}} \nonumber \\
&\quad +\cfrac{p(p-1)2^{p-4}\left(\frac{p}{p-1}\right)^p}{1-(p-1)(p-2)2^{p-4}\left(\frac{p}{p-1}\right)^p\alpha^{2-p}}\mathbb E\int_{t_0}^t\int_Z
|F(\tau,x)|^p m({\rm d }x){\rm d }\tau \nonumber\\
&=:D_1(p)\mathbb E\left(\int_{t_0}^t\int_Z|F(\tau,x)|^2 m({\rm d }x){\rm d }\tau\right)^{\frac{p}{2}}
+D_2(p)\mathbb E\int_{t_0}^t\int_Z|F(\tau,x)|^p m({\rm d }x){\rm d }\tau\nonumber\\
&\leq d_p\Bigg\{\mathbb E\left(\int_{t_0}^t\int_Z|F(\tau,x)|^2 m({\rm d }x){\rm d }\tau\right)^{\frac{p}{2}}
+\mathbb E\int_{t_0}^t\int_Z|F(\tau,x)|^p m({\rm d }x){\rm d }\tau \Bigg\}, \nonumber
\end{align}
where $d_p=\max\{D_1(p),D_2(p)\}$ is continuous w.r.t $p$ and $\alpha\geq 1$ is a constant only dependent on $p$,
which plays the role to ensure $1-(p-1)(p-2)2^{p-4}\left(\frac{p}{p-1}\right)^p\alpha^{2-p}>0$.
\end{enumerate}
\end{remark}

\begin{theorem}\label{prop4.1}
Consider {\rm (\ref{SemiSDE})}. Suppose that {\rm (\ref{expi})} holds and $f\in C(\mathbb{R}\times\mathbb H,\mathbb H)$,
$g\in C(\mathbb{R}\times\mathbb H,L(U,\mathbb H))$, $F,G\in C(\mathbb{R}\times\mathbb H,\mathcal L^2(\nu;\mathbb H))$.
Suppose further that $W$ as well as $N$ are still the same as in Section \ref{levy} and
$f,g,F,G$ satisfy {\rm{(E1$'$)}, {(E2$'$)}}. Then for the constant $p>2$ in {\rm{(E1$'$)} and {(E2$'$)}},
{\rm (\ref{SemiSDE})} has a unique solution in $C_b(\mathbb R,\mathcal L^p(\mathbf{P};\mathbb{H}))$ if
\begin{align*}
\theta_p:=&4^{p-1}K^p\mathcal L^p\Bigg\{\left((1+(2b)^{p-1})\bigg(\frac{2(p-1)}{\omega p}\bigg)^{p-1}
+\left(c_p+(1+2^{p-1})d_p\right)\bigg(\frac{p-2}{\omega p}\bigg)^{\frac{p}{2}-1}\right)\cdot\frac{2}{\omega p}\\
&\quad+(1+2^{p-1})d_p\cdot\frac{1}{\omega p}\Bigg\} <1,
\end{align*}
where $c_p$ and $d_p$ are defined as in \textsl {Remark} {\rm \ref{BDG}}.
\end{theorem}

\begin{proof}
We note that $C_b(\mathbb R,\mathcal L^p(\mathbf{P};\mathbb{H}))$ is a complete metric space
and that $Y\in C_b(\mathbb R,\mathcal L^p(\mathbf{P};\mathbb{H}))$  is a mild solution of (\ref{SemiSDE})
if and only if it satisfies with the following integral equation
\begin{align*}
Y(t)=&\int_{-\infty}^{t}T(t-\tau)f(\tau,Y(\tau)){\rm d }\tau+\int_{-\infty}^{t}T(t-\tau)g(\tau,Y(\tau)){\rm d }W(\tau)\\
&+\int_{-\infty}^{t}\int_{|x|_U<1}T(t-\tau)F(\tau,Y(\tau-),x)\widetilde N({\rm d }\tau,{\rm d }x)\\
&+\int_{-\infty}^{t}\int_{|x|_U\geq1}T(t-\tau)G(\tau,Y(\tau-),x)N({\rm d }\tau,{\rm d }x).
\end{align*}
Consider the operator $\mathcal{S}$ acting on $C_b(\mathbb R,\mathcal L^p(\mathbf{P};\mathbb{H}))$ given by
\begin{align*}
(\mathcal{S}Y)(t):=&\int_{-\infty}^{t}T(t-\tau)f(\tau,Y(\tau)){\rm d }\tau+\int_{-\infty}^{t}T(t-\tau)g(\tau,Y(\tau)){\rm d }W(\tau)\\
&+\int_{-\infty}^{t}\int_{|x|_U<1}T(t-\tau)F(\tau,Y(\tau-),x)\widetilde N({\rm d }\tau,{\rm d }x)\\
&+\int_{-\infty}^{t}\int_{|x|_U\geq1}T(t-\tau)G(\tau,Y(\tau-),x)N({\rm d }\tau,{\rm d }x).
\end{align*}
Since  {\rm (\ref{expi})} holds and $f, g, F, G$ satisfy (E1$'$), (E2$'$), it is immediate to check
that $\mathcal S$ maps $C_b(\mathbb R,\mathcal L^p(\mathbf{P};\mathbb{H}))$ into itself.
Besides if $\mathcal S$ is a contraction mapping on $C_b(\mathbb R,\mathcal L^p(\mathbf{P};\mathbb{H}))$,
then by Banach fixed point theorem it concludes that there exists a unique $\mathcal L^p$-bounded
continuous mild solution to (\ref{SemiSDE}).
We now show that the operator $\mathcal{S}$ is contractive on $C_b(\mathbb R,\mathcal L^p(\mathbf{P};\mathbb{H}))$.
For $Y_1,Y_2\in C_b(\mathbb R,\mathcal L^p(\mathbf{P};\mathbb{H}))$ and $t\in\mathbb R$ we have
\begin{align*}
\mathbb{E}&|(\mathcal{S}Y_1)(t)-(\mathcal{S}Y_2)(t)|^p \\
&\leq4^{p-1}\bigg[\mathbb E\left|\int_{-\infty}^{t}T(t-\tau)(f(\tau,Y_1(\tau))-f(\tau,Y_2(\tau))){\rm d }\tau\right|^p \\
&\qquad+\mathbb E\left|\int_{-\infty}^{t}T(t-\tau)(g(\tau,Y_1(\tau))-g(\tau,Y_2(\tau))){\rm d }W(\tau)\right|^p \\
&\qquad+\mathbb E\bigg|\int_{-\infty}^{t}\int_{|x|_U<1}T(t-\tau)
(F(\tau,Y_1(\tau-),x)-F(\tau,Y_2(\tau-),x))\widetilde{N}({\rm d }\tau,{\rm d }x)\bigg|^p\\
&\qquad+\mathbb E\bigg|\int_{-\infty}^{t}\int_{|x|_U\geq1}T(t-\tau)
(G(\tau,Y_1(\tau-),x)-G(\tau,Y_2(\tau-),x))N({\rm d }\tau,{\rm d }x)\bigg|^p\bigg]\\
&=:4^{p-1}[Z_1(t)+Z_2(t)+Z_3(t)+Z_4(t)].
\end{align*}
The first term can be estimated by H\"older's inequality as follows
\begin{align*}
Z_1(t)&\leq\mathbb{E}\left(\int_{-\infty}^t\|T(t-\tau)\|\cdot|f(\tau,Y_1(\tau))-f(\tau,Y_2(\tau))|{\rm d }\tau\right)^p \\
&\leq \mathbb{E}\left(\int_{-\infty}^{t}K{\rm e}^{-\omega(t-\tau)}\mathcal L|Y_1(\tau)-Y_2(\tau)|{\rm d }\tau\right)^p \\
&\leq K^{p}\mathcal L^{p}\left(\int_{-\infty}^{t}{\rm e}^{\frac{-\omega(t-\tau)}{2}\cdot\frac{p}{p-1}}{\rm d }\tau\right)^{p-1}\cdot
\mathbb{E}\int_{-\infty}^t{\rm e}^{\frac{-\omega(t-\tau)}{2}\cdot p}|Y_1(\tau)-Y_2(\tau)|^{p}{\rm d }\tau \\
&\leq K^{p}\mathcal L^{p}\left(\frac{2(p-1)}{\omega p}\right)^{p-1}\cdot\int_{-\infty}^{t}{\rm e}^{\frac{-\omega p(t-\tau)}{2}}
\mathbb{E}|Y_1(\tau)-Y_2(\tau)|^{p}{\rm d }\tau.
\end{align*}
By (\ref{cp}) and H\"older's inequality, for the second term we have
\begin{align*}
Z_2(t)&\leq c_p\left(\mathbb E\int_{-\infty}^t\|T(t-\tau)(g(\tau,Y_1(\tau))-g(\tau,Y_2(\tau)))
\mathcal Q^\frac{1}{2}\|_{L_2(U,\mathbb H)}^2{\rm d}\tau\right)^{\frac{p}{2}} \\
&\leq K^{p}\mathcal L^{p}c_p\left(\mathbb E\int_{-\infty}^{t}{\rm e}^{-2\omega(t-\tau)}|Y_1(\tau)-Y_2(\tau)|^2{\rm d }\tau\right)^\frac{p}{2}\\
&\leq K^p\mathcal L^{p}c_p\left(\int_{-\infty}^t{\rm e}^{-\omega(t-\tau)\cdot\frac{p}{p-2}}{\rm d}\tau\right)^{\frac{p-2}{2}}\cdot
\mathbb E\int_{-\infty}^{t}{\rm e}^{-\omega(t-\tau)\cdot\frac{p}{2}}|Y_1(\tau)-Y_2(\tau)|^{p}{\rm d }\tau\ \\
&\leq K^{p}\mathcal{L}^{p}c_p\left(\frac{p-2}{\omega p}\right)^{\frac{p}{2}-1}\cdot\int_{-\infty}^{t}
{\rm e}^{\frac{-\omega p(t-\tau)}{2}}\mathbb E|Y_1(\tau)-Y_2(\tau)|^{p}{\rm d }\tau.
\end{align*}
As to the third and the forth terms, thanks to Kunita's first inequality (\ref{kunita}) and H\"older's inequality,
we get
\begin{align*}
Z_3(t)
\leq&d_p\Bigg\{\mathbb E\left(\int_{-\infty}^t\int_{|x|_U<1}|T(t-\tau)(F(\tau,Y_1(\tau-),x)-
F(\tau,Y_2(\tau-),x))|^2\nu({\rm d }x){\rm d }\tau\right)^\frac{p}{2}\\
&\quad+\mathbb E\int_{-\infty}^t\int_{|x|_U<1}|T(t-\tau)(F(\tau,Y_1(\tau-),x)-F(\tau,Y_2(\tau-),x))|^p\nu({\rm d }x){\rm d }\tau\Bigg\}\\
\leq&d_pK^p\mathcal L^p\Bigg\{\mathbb E\bigg(\int_{-\infty}^t{\rm e}^{-2\omega(t-\tau)}|Y_1(\tau)-Y_2(\tau)|^2{\rm d }\tau\bigg)^\frac{p}{2}
+\mathbb E\int_{-\infty}^t{\rm e}^{-\omega p(t-\tau)}|Y_1(\tau)-Y_2(\tau)|^p{\rm d }\tau\Bigg\} \\
\leq&d_pK^p\mathcal L^p\Bigg\{\bigg(\int_{-\infty}^t{\rm e}^{-\omega(t-\tau)\cdot\frac{p}{p-2}}{\rm d }\tau\bigg)^\frac{p-2}{2}
\cdot\mathbb E\int_{-\infty}^t{\rm e}^{-\omega(t-\tau)\cdot\frac{p}{2}}|Y_1(\tau)-Y_2(\tau)|^p{\rm d }\tau\\
&\quad+\mathbb E\int_{-\infty}^t{\rm e}^{-\omega p(t-\tau)}|Y_1(\tau)-Y_2(\tau)|^p{\rm d }\tau\Bigg\}\\
\leq&d_pK^p\mathcal L^p\Bigg\{\bigg(\frac{p-2}{\omega p}\bigg)^\frac{p-2}{2}\cdot\int_{-\infty}^t
{\rm e}^\frac{-\omega p(t-\tau)}{2}\mathbb E|Y_1(\tau)-Y_2(\tau)|^p{\rm d }\tau\\
&\quad+\int_{-\infty}^t{\rm e}^{-\omega p(t-\tau)}\mathbb E|Y_1(\tau)-Y_2(\tau)|^p{\rm d }\tau\Bigg\}
\end{align*}
and
\begin{align*}
Z_4(t)\leq&2^{p-1}\mathbb E\bigg|\int_{-\infty}^{t}\int_{|x|_U\geq1}T(t-\tau)(G(\tau,Y_1(\tau-),x)-G(\tau,Y_2(\tau-),x))
\widetilde{N}({\rm d }\tau,{\rm d }x)\bigg|^p \\
&+2^{p-1}\mathbb E\bigg|\int_{-\infty}^{t}\int_{|x|_U\geq1}T(t-\tau)(G(\tau,Y_1(\tau-),x)-G(\tau,Y_2(\tau-),x))
\nu({\rm d }x){\rm d }\tau\bigg|^p\\
\leq&2^{p-1}d_p\Bigg\{\mathbb E\left(\int_{-\infty}^t\int_{|x|_U\geq1}|T(t-\tau)(G(\tau,Y_1(\tau-),x)
-G(\tau,Y_2(\tau-),x))|^2\nu({\rm d }x){\rm d }\tau\right)^{\frac{p}{2}} \\
&\quad+\mathbb E\int_{-\infty}^t\int_{|x|_U\geq1}|T(t-\tau)(G(\tau,Y_1(\tau-),x)
-G(\tau,Y_2(\tau-),x))|^p\nu({\rm d }x){\rm d }\tau\Bigg\} \\
&+2^{p-1}\mathbb E\bigg|\int_{-\infty}^t\int_{|x|_U\geq1}T(t-\tau)(G(\tau,Y_1(\tau-),x)
-G(\tau,Y_2(\tau-),x))\nu({\rm d }x){\rm d }\tau\bigg|^p \\
\leq&2^{p-1}d_pK^p\mathcal L^p\Bigg\{\bigg(\frac{p-2}{\omega p}\bigg)^\frac{p-2}{2}\cdot\int_{-\infty}^t
{\rm e}^\frac{-\omega p(t-\tau)}{2}\mathbb E|Y_1(\tau)-Y_2(\tau)|^p{\rm d }\tau \\
&\quad+\int_{-\infty}^t{\rm e}^{-\omega p(t-\tau)}\mathbb E|Y_1(\tau)-Y_2(\tau)|^p{\rm d }\tau\Bigg\}\\
&+2^{p-1}K^p\mathcal L^p\bigg(\int_{-\infty}^t\int_{|x|_U\geq1}{\rm e}^{\frac{-\omega p(t-\tau)}{2(p-1)}}
\nu({\rm d }x){\rm d }\tau\bigg)^{p-1}\cdot \int_{-\infty}^t{\rm e}^{\frac{-\omega p(t-\tau)}{2}}
\mathbb E|Y_1(\tau)-Y_2(\tau)|^p{\rm d }\tau\\
\leq&2^{p-1}K^p\mathcal L^p\Bigg\{\bigg(d_p\bigg(\frac{p-2}{\omega p}\bigg)^\frac{p-2}{2}+
\bigg(\frac{2b(p-1)}{\omega p}\bigg)^{p-1}\bigg) \\
&\quad\cdot\int_{-\infty}^t{\rm e}^\frac{-\omega p(t-\tau)}{2}\mathbb E|Y_1(\tau)-Y_2(\tau)|^p{\rm d }\tau
+d_p\int_{-\infty}^{t}{\rm e}^{-\omega p(t-\tau)}\mathbb E|Y_1(\tau)-Y_2(\tau)|^p{\rm d }\tau\Bigg\}.
\end{align*}
Therefore, for each $t\in \mathbb R$,
\begin{align*}
\mathbb E&|(\mathcal{S}Y_1)(t)-(\mathcal{S}Y_2)(t)|^p \\
&\leq 4^{p-1}K^p\mathcal L^p\Bigg\{\left((1+(2b)^{p-1})\bigg(\frac{2(p-1)}{\omega p}\bigg)^{p-1}
+\left(c_p+(1+2^{p-1})d_p\right)\bigg(\frac{p-2}{\omega p}\bigg)^{\frac{p}{2}-1}\right)\cdot\frac{2}{\omega p}\\
&\qquad+(1+2^{p-1})d_p\cdot\frac{1}{\omega p}\Bigg\}
\cdot\sup_{\tau\in\mathbb R}\mathbb E|Y_1(\tau)-Y_2(\tau)|^p.
\end{align*}
That is
\begin{equation*}
\sup_{t\in\mathbb R}\mathbb E|(\mathcal{S}Y_1)(t)-(\mathcal{S}Y_2)(t)|^p\leq\theta_p\sup_{t\in\mathbb R}\mathbb E|Y_1(t)-Y_2(t)|^p.
\end{equation*}
It implies that
\begin{equation*}
\|\mathcal SY_1-\mathcal SY_2\|_\infty^p\leq\theta_p\|Y_1-Y_2\|_\infty^p.
\end{equation*}
Since $\theta_p<1$, the operator $\mathcal{S}$ is a contraction mapping on $C_b(\mathbb R,\mathcal L^p(\mathbf{P};\mathbb{H}))$.
Thus there exists a unique $\xi\in C_b(\mathbb R,\mathcal L^p(\mathbf{P};\mathbb{H}))$ satisfying
$\mathcal{S}\xi=\xi$, which is the unique $\mathcal{L}^p$-bounded
 solution of (\ref{SemiSDE}).
\end{proof}

\begin{remark}\label{RLp}\rm
Note that the contraction constant $\theta_p$ is continuous in $p$ for $p>2$. According to Remark \ref{BDG},
we have $c_p=1$ and $d_p=2$ when $p=2$, so
\[
\lim_{p\to 2^+} \theta_p=\frac{4K^2\mathcal L^2}{\omega^2}(1+10\omega+2b).
\]
\end{remark}

\begin{lemma}\label{lB} {\rm {(\cite{CL})}}
Let $u,f\in C_{b}(\mathbb R,\mathbb
R_{+})$ and $\omega >a \ge 0$, then the following statements hold.
\begin{enumerate}
\item If
\begin{equation*}\label{eqB1}
u(t)\le \int_{-\infty}^{t}{\rm e}^{-\omega (t-\tau)}(a u(\tau) +
f(\tau)) {\rm d}\tau
\end{equation*}
for any $t\in \mathbb R$, then
\begin{equation*}\label{eqB2}
u(t)\le \int_{-\infty}^{t}{\rm e}^{-\alpha(t-\tau)}f(\tau) {\rm d}\tau,
\end{equation*}
where $\alpha:=\omega -a$.
\item If $l>k>0$, then
\begin{equation*}\label{eqB3}
\max\limits_{|t|\le k}u(t)\le
\frac{{\rm e}^{\alpha k}{\rm e}^{-\alpha l}}{\alpha}\sup\limits_{t\in\mathbb
R}f(t)+\frac{1-{\rm e}^{-\alpha k}{\rm e}^{-\alpha l}}{\alpha} \max\limits_{|t|\le l}f(t).
\end{equation*}
\end{enumerate}
\end{lemma}

\begin{theorem}\label{thmtwo}
Consider {\rm(\ref{SemiSDE})}. Suppose that the semigroup $\{T(t)\}_{t\geq0}$ acting on $\mathbb{H}$ is dissipative
such that {\rm(\ref{expi})} holds; $f\in C(\mathbb{R}\times\mathbb H,\mathbb H)$, $g\in C(\mathbb{R}\times\mathbb H,L(U,\mathbb H))$, and
$F, G\in C(\mathbb{R}\times\mathbb H,\mathcal L^2(\nu;\mathbb H))$. Assume that $W$ as well as $N$ are
the same as in Section \ref{levy} and $f, g, F, G$ satisfy {\rm(E1)}, {\rm(E2)}.
If $\mathcal L<\cfrac{\omega}{2K\sqrt{1+2\omega+2b}}$, then {\rm(\ref{SemiSDE})} has a unique solution $\xi\in C(\mathbb R,B[0,r])$, where
$B[0,r]:=\{Y\in\mathcal{L}^2(\mathbf {P};\mathbb H):\|Y\|_{\mathcal L^2(\mathbf{P}; \mathbb{H})}\leq r\}$ and
\begin{align}\label{four 1}
r=\cfrac{2KA_0\sqrt{1+2\omega+2b}}{\omega-2K\mathcal L\sqrt{1+2\omega+2b}}.
\end{align}
Moreover, suppose that $f,g,F,G$ satisfy {\rm(E1$'$)}, {\rm(E2$'$)}, {\rm(E3)} additionally. Then if
\begin{equation}\label{L}
\mathcal{L}<\bigg\{\cfrac{\omega}{2K\sqrt{2+4\omega+4b}}\wedge \cfrac{\omega}{2K\sqrt{1+10\omega+2b}}\bigg\},
\end{equation}
we have $\mathfrak M^{u}_{(f,g,F,G)} \subseteq \tilde{\mathfrak M}^{u}_{\xi}$, where $\tilde{\mathfrak M}^{u}_{\xi}$
means the set of all sequences $\{t_n\}$ such that $\xi(t+t_n)$ converges in distribution uniformly in $t\in\mathbb R$; if
\begin{equation}\label{L11}
\mathcal{L}<\cfrac{\omega}{2K\sqrt{2+8\omega+4b}},
\end{equation}
the solution $\xi$ is strongly compatible in distribution.
\end{theorem}

\begin{proof}
{\bf Step 1. There exists a unique solution $\xi\in C(\mathbb{R}, B[0,r])$.}
Note that $C(\mathbb R,B[0,r])$ is a complete metric space. Let $\phi\in C(\mathbb R,B[0,r])$ and we put
$h_1(t):=f(t,\phi(t))$, $h_2(t):=g(t,\phi(t))$, $h_3(t,x):=F(t,\phi(t-),x)$, $h_4(t,x):=G(t,\phi(t-),x)$.
Since $f$ satisfies {\rm(E1)} and {\rm(E2)}, we have for $t\in\mathbb R$
\begin{equation}\label{4.1}
\|h_1(t)\|_{\mathcal L^2(\mathbf{P}; \mathbb{H})}=\|f(t,\phi(t))\|_{\mathcal L^2(\mathbf{P}; \mathbb{H})}
=\|f(t,0)+f(t,\phi(t))-f(t,0)\|_{\mathcal L^2(\mathbf{P}; \mathbb{H})}\leq A_0+\mathcal Lr.
\end{equation}
Similarly, we obtain for $t\in\mathbb R$
\begin{equation}\label{4.2}
\|h_2(t)\mathcal Q^\frac{1}{2}\|_{\mathcal L^2(\mathbf P;L_2(U,\mathbb H))}\leq A_0+\mathcal Lr,
\end{equation}
\begin{equation}\label{4.3}
\|h_3(t,x)\|_{\mathcal L^2(\mathbf P_\nu;\mathbb H)}\leq  A_0+\mathcal Lr,
\end{equation}
\begin{equation}\label{4.4}
 \|h_4(t,x)\|_{\mathcal L^2(\mathbf P_\nu;\mathbb H)}\leq  A_0+\mathcal Lr.
\end{equation}
According to Theorem \ref{thmone}, the equation
\begin{align*}
{\rm d }Z(t)=&(AZ(t)+h_1(t)){\rm d }t+h_2(t){\rm d }W(t)\\
&+\int_{|x|_U<1}h_3(t,x)\widetilde{N}({\rm d }t,{\rm d }x)+\int_{|x|_U\geq1}h_4(t,x)N({\rm d }t,{\rm d }x)
\end{align*}
admits a unique solution $\psi\in C_b(\mathbb R,\mathcal{L}^2(\mathbf{P};\mathbb{H}))$. Besides by \eqref{item22} it obeys the estimate
\begin{align}\label{four one}
\|\psi\|_\infty^2\leq \frac{K^2}{\omega^2}\left(4\|h_1\|_\infty^2+2\omega\|h_2\mathcal Q^\frac{1}{2}\|_\infty^2
+2\omega\|h_3\|_\infty^2+(4\omega+8b)\|h_4\|_\infty^2\right).
\end{align}
From (\ref{4.1})-(\ref{four one}) and (\ref{four 1}), we get
\begin{align*}
\|\psi\|_\infty
&\leq\frac{2K}{\omega}\sqrt{1+2\omega+2b}(A_0+\mathcal Lr)=r.
\end{align*}
Thus $\psi\in C(\mathbb R,B[0,r])$. Let $\Phi(\phi):=\psi$. It follows from the above argument that
$\Phi:C(\mathbb R,B[0,r])\to C(\mathbb R,B[0,r])$.

Now we show that $\Phi$ is a contraction operator. For any $\psi_1,\psi_2\in C(\mathbb R,B[0,r])$,
note that $\psi_1-\psi_2=\Phi(\phi_1)
-\Phi(\phi_2)$ is the unique solution in $C_b(\mathbb{R},\mathcal{L}^2(\mathbf{P};\mathbb{H}))$ of the
equation
\begin{align*}
{\rm d }u(t)=&(Au(t)+f(t,\phi_1(t))-f(t,\phi_2(t))){\rm d }t +(g(t,\phi_1(t))-g(t,\phi_2(t))){\rm d }W(t)\\
&+\int_{|x|_U<1}(F(t,\phi_1(t-),x)-F(t,\phi_2(t-),x))\widetilde{N}({\rm d }t,{\rm d }x)\\
&+\int_{|x|_U\geq1}(G(t,\phi_1(t-),x)-G(t,\phi_2(t-),x))N({\rm d }t,{\rm d }x).
\end{align*}
According to (\ref{item2}), we have the following estimate
\begin{align}\label{4.7}
&\|\Phi(\phi_1)-\Phi(\phi_2)\|_{\infty}^2 \\
&\quad\leq\frac{K^2}{\omega^2}
\bigg(4\sup_{t\in\mathbb R}\mathbb E|f(t,\phi_1(t))-f(t,\phi_2(t))|^2+2\omega\sup_{t\in\mathbb R}
\mathbb E\|(g(t,\phi_1(t))-g(t,\phi_2(t)))\mathcal{Q}^\frac{1}{2}\|_{L_2(U,\mathbb H)}^2 \nonumber\\
&\quad\qquad+2\omega\sup_{t\in\mathbb{R}}\int_{|x|_U<1}\mathbb E|F(t,\phi_1(t-),x)-F(t,\phi_2(t-),x)|^2\nu({\rm d }x) \nonumber\\
&\quad\qquad+(4\omega+8b)\sup_{t\in\mathbb{R}}\int_{|x|_U\geq1}\mathbb E|G(t,\phi_1(t-),x)
-G(t,\phi_2(t-),x)|^2\nu({\rm d }x)\bigg)\nonumber\\
&\quad\leq\frac{4K^2\mathcal{L}^2}{\omega^2}(1+2\omega+2b)\|\phi_1-\phi_2\|_{\infty}^2\nonumber\\
&\quad:=\theta_2\|\phi_1-\phi_2\|_{\infty}^2.\nonumber
\end{align}
By our assumption on $\mathcal{L}$ we have $\theta_2<1$,
so $\Phi$ is a contraction operator on $C(\mathbb R,B[0,r])$. On the basis of Banach fixed point theorem,
there exists a unique function $\xi\in C(\mathbb R,B[0,r])$ such that $\Phi(\xi)=\xi$.

{\bf Step 2. $\mathfrak M^{u}_{(f,g,F,G)} \subseteq \tilde{\mathfrak M}^{u}_{\xi}$.}
Let $\{t_n\}\in \mathfrak M_{(f,g,F,G)}^{u}$. Then there exists
$(\tilde{f},\tilde{g},\tilde{F},\tilde{G})\in H(f,g,F,G)$ such that for any $r>0$
\begin{equation}\label{4.8}
\sup_{t\in\mathbb{R},|Y|\leq r}|f(t+t_n,Y)-\widetilde{f}(t,Y)|\to 0,
\end{equation}
\begin{equation}\label{4.9}
\sup_{t\in\mathbb{R},|Y|\leq r}\|(g(t+t_n,Y)-\widetilde{g}(t,Y))\mathcal{Q}^\frac{1}{2}\|_{L_2(U,\mathbb H)}\to 0,
\end{equation}
\begin{equation}\label{4.9.1}
\sup_{t\in\mathbb{R},|Y|\leq r}\int_{|x|_U<1}|F(t+t_n,Y,x)-\widetilde{F}(t,Y,x)|^2\nu({\rm d }x)\to 0,
\end{equation}
and
\begin{equation}\label{4.9.2}
\sup_{t\in\mathbb{R},|Y|\leq r}\int_{|x|_U\geq1}|G(t+t_n,Y,x)-\widetilde{G}(t,Y,x)|^2\nu({\rm d }x)\to 0
\end{equation}
as $n\to\infty$.
Consider equations
\begin{align}\label{4.10}
{\rm d }Y(t)=&(AY(t)+f^{t_n}(t,Y(t))){\rm d }t+g^{t_n}(t,Y(t)){\rm d }W(t)\\
&+\int_{|x|_U<1}F^{t_n}(t,Y(t-),x)\widetilde{N}({\rm d }t,{\rm d }x)  \nonumber\\
&+\int_{|x|_U\geq1}G^{t_n}(t,Y(t-),x)N({\rm d }t,{\rm d }x) \qquad (n\in\mathbb{N})\nonumber
\end{align}
and
\begin{align}\label{4.11}
{\rm d }Y(t)=&(AY(t)+\widetilde{f}(t,Y(t))){\rm d }t+\widetilde{g}(t,Y(t)){\rm d }W(t) \\
&+\int_{|x|_U<1}\widetilde{F}(t,Y(t-),x)\widetilde{N}({\rm d }t,{\rm d }x)
+\int_{|x|_U\geq1}\widetilde{G}(t,Y(t-),x)N({\rm d }t,{\rm d }x).\nonumber
\end{align}
Since $(f^{t_n},g^{t_n},F^{t_n},G^{t_n})$ and $(\widetilde{f},\widetilde{g},\widetilde{F},\widetilde{G})$
satisfy (E1) and (E2) (see Remark     \ref{H(f,g,F,G)}), (\ref{4.10})
(respectively, (\ref{4.11})) has a unique solution $\xi_n\in C(\mathbb{R},B[0,r])$ (respectively,
$\widetilde{\xi}\in C(\mathbb{R},B[0,r])$). Next we show that $\{\xi_n(t)\}$ converges to $\widetilde{\xi}(t)$
in $\mathcal{L}^2$-norm uniformly in $t\in\mathbb{R}$. Let
\begin{align*}
&a_n(t)=f^{t_n}(t,\xi_n(t)), \quad\quad\quad\quad  b_n(t)=g^{t_n}(t,\xi_n(t)),\\
&c_n(t,x)=F^{t_n}(t,\xi_n(t-),x),\quad d_n(t,x)=G^{t_n}(t,\xi_n(t-),x), \\
&\widetilde{a}(t)=\widetilde{f}(t,\widetilde{\xi}(t)), \quad\quad\qquad\qquad \widetilde{b}(t)=\widetilde{g}(t,\widetilde{\xi}(t)), \\
&\widetilde{c}(t,x)=\widetilde{F}(t,\widetilde{\xi}(t-),x),\quad\qquad \widetilde{d}(t,x)=\widetilde{G}(t,\widetilde{\xi}(t-),x).
\end{align*}
Then $\xi_n(n\in \mathbb N)$ is the unique solution from $C(\mathbb{R},B[0,r])$ of equation
\begin{align*}
{\rm d }Y(t)=&(AY(t)+a_n(t)){\rm d }t+b_n(t){\rm d }W(t)\\
&+\int_{|x|_U<1}c_n(t,x)\widetilde{N}({\rm d }t,{\rm d }x)+\int_{|x|_U\geq1}d_n(t,x)N({\rm d }t,{\rm d }x)\quad(n\in \mathbb N)
\end{align*}
and $\widetilde{\xi}$ is the unique solution from $C(\mathbb{R},B[0,r])$ of equation
\begin{align*}
{\rm d }Y(t)=&(AY(t)+\widetilde{a}(t)){\rm d }t+\widetilde{b}(t){\rm d }W(t) \\
&+\int_{|x|_U<1}\widetilde{c}(t,x)\widetilde{N}({\rm d }t,{\rm d }x)+\int_{|x|_U\geq1}\widetilde{d}(t,x)N({\rm d }t,{\rm d }x).
\end{align*}
Note that $\phi_{n}:=\xi_{n}-\tilde{\xi}$
is the unique solution from $C(\mathbb R, B[0,2r])$ of equation
\begin{align}\label{4.12}
{\rm d }Y(t)=&(AY(t)+a_n(t)-\widetilde{a}(t)){\rm d }t+(b_n(t)-\widetilde{b}(t)){\rm d }W(t) \\
&+\int_{|x|_U<1}(c_n(t,x)-\widetilde{c}(t,x))\widetilde{N}({\rm d }t,{\rm d }x) \nonumber\\
&+\int_{|x|_U\geq1}(d_n(t,x)-\widetilde{d}(t,x))N({\rm d }t,{\rm d }x) \quad (n\in \mathbb N), \nonumber
\end{align}
where $a_n-\widetilde{a}\in C_b(\mathbb{R},\mathcal{L}^2(\mathbf{P};\mathbb{H}))$,
$b_n-\widetilde{b}\in C_b(\mathbb R,\mathcal L^2(\mathbf P;L(U,\mathbb H)))$,
$c_n-\widetilde{c}\in C_b(\mathbb R,\mathcal L^2(\mathbf P_\nu;\mathbb H))$,
$d_n-\widetilde{d}\in C_b(\mathbb R,\mathcal L^2(\mathbf P_\nu;\mathbb H))$.
By (\ref{item22}), we have
\begin{align}\label{4.13}
\|\phi_n\|_\infty^2 \leq&\frac{K^2}{\omega^2}\Big(4\|a^{t_n}-\widetilde a\|_{\infty}^2
+2\omega\|(b^{t_n}-\widetilde b)\mathcal Q^\frac{1}{2}\|_{\infty}^2 +2\omega\|c^{t_n}-\widetilde c\|_\infty^2\\
&\quad
+(4\omega+8b)\|d^{t_n}-\widetilde d\|_\infty^2\Big). \nonumber
\end{align}
Since $(f^{t_n},g^{t_n},F^{t_n},G^{t_n})(n\in N)$ and $(\widetilde{f},\widetilde{g},\widetilde{F},\widetilde{G})$ satisfy (E1) and (E2), and
$\xi_n,\widetilde{\xi}\in C(\mathbb{R},B[0,r])$ $(n\in \mathbb N)$, we have
\begin{align}\label{4.14}
\mathbb E|a_n(\tau)-\widetilde a(\tau)|^2 &=\mathbb E|f^{t_n}(\tau,\xi_n(\tau))
-f^{t_n}(\tau,\widetilde\xi(\tau))+f^{t_n}(\tau,\widetilde\xi(\tau))-\widetilde f(\tau,\widetilde\xi(\tau))|^2 \\
&\leq2\mathbb E|f^{t_n}(\tau,\xi_n(\tau))-f^{t_n}(\tau,\widetilde\xi(\tau))|^2
+2\mathbb E|f^{t_n}(\tau,\widetilde\xi(\tau))-\widetilde f(\tau,\widetilde\xi(\tau))|^2  \nonumber\\
&\leq2\mathcal L^2\mathbb E|\xi_n(\tau)-\widetilde\xi(\tau)|^2+2\sup_{\tau\in\mathbb R}
\mathbb E|f^{t_n}(\tau,\widetilde\xi(\tau))-\widetilde f(\tau,\widetilde\xi(\tau))|^2  \nonumber\\
&\leq2\mathcal L^2\|\phi_n\|_{\infty}^2+2\sup_{\tau\in\mathbb R}\mathbb EI_1^2(n,\tau),\nonumber
\end{align}
where
\begin{equation*}
I_1^2(n,\tau):=|f^{t_n}(\tau,\widetilde{\xi}(\tau))-\widetilde{f}(\tau,\widetilde{\xi}(\tau))|^2.
\end{equation*}
Similarly, we get
\begin{equation}\label{4.15 1}
\mathbb E\|(b_n(\tau)-\widetilde{b}(\tau))\mathcal{Q}^\frac{1}{2}\|_{L_2(U,\mathbb H)}^2\leq2\mathcal{L}^2\|\phi_n\|_\infty^2
+2\sup_{\tau\in\mathbb{R}}\mathbb EI_2^2(n,\tau),
\end{equation}
where
\begin{equation*}
I_2^2(n,\tau):=\|(g^{t_n}(\tau,\widetilde{\xi}(\tau))
-\widetilde{g}(\tau,\widetilde{\xi}(\tau)))\mathcal Q^\frac{1}{2}\|_{L_2(U,\mathbb{H})}^2;
\end{equation*}

\begin{align}\label{4.15 2}
\int_{|x|_U<1}\mathbb E|c_n(\tau,x)-\widetilde c(\tau,x)|^2\nu({\rm d }x)
\leq2\mathcal L^2\|\phi_n\|_\infty^2+2\sup_{\tau\in\mathbb R}\mathbb E I_3^2(n,\tau),
\end{align}
where
\begin{equation*}
I_3^2(n,\tau):=\int_{|x|_U<1}|F^{t_n}(\tau,\widetilde{\xi}(\tau-),x)-
\widetilde{F}(\tau,\widetilde{\xi}(\tau-),x)|^2\nu({\rm d }x);
\end{equation*}
\begin{align}\label{4.15 3}
\int_{|x|_U\geq1}\mathbb E|d_n(\tau,x)-\widetilde d(\tau,x)|^2\nu({\rm d }x)
\leq2\mathcal L^2\|\phi_n\|_\infty^2+2\sup_{\tau\in\mathbb R}\mathbb E I_4^2(n,\tau),
\end{align}
where
\begin{equation*}
I_4^2(n,\tau):=\int_{|x|_U\geq1}|G^{t_n}(\tau,\widetilde{\xi}(\tau-),x)
-\widetilde{G}(\tau,\widetilde{\xi}(\tau-),x)|^2\nu({\rm d }x).
\end{equation*}
In view of (\ref{4.13})-(\ref{4.15 3}), we have
\begin{align}\label{4.16}
c\|\phi_n\|_\infty^2\leq&\frac{8K^2}{\omega^2}\sup_{\tau\in\mathbb{R}}\mathbb EI_1^2(n,\tau)
+\frac{4K^2}{\omega}\sup_{\tau\in\mathbb{R}}\mathbb EI_2^2(n,\tau) \\
&+\frac{4K^2}{\omega}\sup_{\tau\in\mathbb{R}}\mathbb EI_3^2(n,\tau)
+\left(\frac{8K^2}{\omega}+\frac{16K^2b}{\omega^2}\right)
\sup_{\tau\in\mathbb{R}}\mathbb EI_4^2(n,\tau),\nonumber
\end{align}
where $c=1-\frac{8K^2\mathcal L^2}{\omega^2}(1+2\omega+2b)$. By our assumption on $\mathcal{L}$ (i.e. (\ref{L})),
the coefficient of $\|\phi_n\|_\infty^2$ is positive.

Now we show that the family $\{|\widetilde \xi(\tau)|^2:\tau\in\mathbb R\}$ is uniformly integrable.
Indeed, by (\ref{4.7}) and Remark \ref{H(f,g,F,G)}-(i), for $p=2$, the contraction constant $\theta_2$ for (\ref{4.11}) is
\begin{align*}
\theta_2=\cfrac{4K^2\mathcal{L}^2}{\omega^2}(1+2\omega+2b).
\end{align*}
Comparing to Remark \ref{RLp}, we have
\begin{align*}
\lim_{p\to 2^+} \theta_p=\theta_2+\frac{32K^2\mathcal{L}^2}{\omega^2}.
\end{align*}
We also note that by our assumption on $\mathcal L$ (i.e. (\ref{L})), $\lim\limits_{p\to 2^+}\theta_p<1$.
So by Proposition \ref{prop4.1}, (\ref{4.11}) has a unique $\mathcal{L}^p$-bounded solution for some $p>2$.
This $\mathcal L^p$-bounded solution is exactly the unique $\mathcal{L}^2$-bounded solution $\tilde\xi$ of (\ref{4.11}).
Then the family $\{|\widetilde{\xi}(\tau)|^2:\tau\in\mathbb{R}\}$ is uniformly integrable. Hence by (E1) and (E2),
the families $\{I_i^2(n,\tau): n\in\mathbb N, \tau\in \mathbb R\}$ $(i=1,2,3,4)$
are uniformly integrable.

  By (\ref{4.16}), we have
\begin{align*}
c\lim_{n\to\infty}\|\phi_n\|_\infty^2 &\leq\frac{8K^2}{\omega^2}\lim_{n\to\infty}
\mathbb E\sup_{\tau\in\mathbb{R}}I_1^2(n,\tau)+\frac{4K^2}{\omega}\lim_{n\to\infty}\mathbb E\sup_{\tau\in\mathbb{R}}I_2^2(n,\tau) \\
&\quad+\frac{4K^2}{\omega}\lim_{n\to\infty}\mathbb E\sup_{\tau\in\mathbb{R}}I_3^2(n,\tau)
+\left(\frac{8K^2}{\omega}+\frac{16K^2b}{\omega^2}\right)
\lim_{n\to\infty}\mathbb E\sup_{\tau\in\mathbb{R}}I_4^2(n,\tau)\\
&\leq\frac{8K^2}{\omega^2}\mathbb E\lim_{n\to\infty}
\sup_{\tau\in\mathbb{R}}I_1^2(n,\tau)+\frac{4K^2}{\omega}\mathbb E
\lim_{n\to\infty} \sup_{\tau\in\mathbb{R}}I_2^2(n,\tau) \\
&\quad+\frac{4K^2}{\omega}\mathbb E \lim_{n\to\infty} \sup_{\tau\in\mathbb{R}}I_3^2(n,\tau)
+\left(\frac{8K^2}{\omega}+\frac{16K^2b}{\omega^2}\right)
\mathbb E \lim_{n\to\infty} \sup_{\tau\in\mathbb{R}}I_4^2(n,\tau)\\
&=0,
\end{align*}
where the second inequality holds by (\ref{4.8})-(\ref{4.9.2}) and Vitali convergence theorem.
Then we get the required result, i.e. $\xi_n(t)\to\widetilde\xi(t)$ uniformly in $t\in\mathbb{R}$ in $\mathcal{L}^2$-norm.

  On the one hand $\mathcal{L}^2$-convergence implies convergence in distribution, so we have
$\xi_n(t)\to\widetilde\xi(t)$ in distribution uniformly on $\mathbb R$. On the other hand,
$\xi(t+t_n)$ satisfies the equation
\begin{align*}
\xi(t+t_n)=&\int_{-\infty}^{t}T(t-\tau)f(\tau+t_n,\xi(\tau+t_n)){\rm d }\tau\\
&+\int_{-\infty}^{t}T(t-\tau)g(\tau+t_n,\xi(\tau+t_n)){\rm d }W_n(\tau) \\
&+\int_{-\infty}^{t}\int_{|x|_U<1}T(t-\tau)F(\tau+t_n,\xi(\tau+t_n-),x)\widetilde{N}_n({\rm d }\tau,{\rm d }x)\\
&+\int_{-\infty}^{t}\int_{|x|_U\geq1}T(t-\tau)G(\tau+t_n,\xi(\tau+t_n-),x)N_n({\rm d }\tau,{\rm d }x)
\end{align*}
where $W_n(\tau):=W(\tau+t_n)-W(t_n)$, $N_n(\tau,x):=N(\tau+t_n,x)-N(t_n,x)$ and
$\widetilde {N}_n(\tau,x):=\widetilde {N}(\tau+t_n,x)-\widetilde {N}(t_n,x)$, $\tau\in\mathbb R$.
For each $n\in \mathbb N$, by Remark \ref{Lp}, $W_n$ is a $\mathcal Q$-Wiener process with the same law
as $W$ and $N_n$ has the same law as $N$ with compensated Poisson random measure $\widetilde{N}_n$.
Thus $\xi_n(t)$ and $\xi(t+t_n)$ share the same distribution on $\mathbb H$, and this implies
$\xi(t+t_n)\to\widetilde\xi(t)$ in distribution uniformly in $t\in\mathbb{R}$. So we have
$\{t_n\}\in \mathfrak{\tilde{M}}_{\xi}^u$.

{\bf Step 3. The solution $\xi$ is strongly compatible in distribution.} Let $\{t_n\}\in \mathfrak M_{(f,g,F,G)}$.
Then there exists
$(\widetilde{f},\widetilde{g},\widetilde{F},\widetilde{G})\in H(f,g,F,G)$ such that for any $r,l>0$
\begin{equation}\label{eq1}
\sup\limits_{|t|\leq l,|Y|\leq r}|f(t+t_n,Y)-\widetilde f(t,Y)|\to0,
\end{equation}
\begin{equation}\label{eq2}
\sup\limits_{|t|\leq l,|Y|\leq r}
\|(g(t+t_n,Y)-\widetilde g(t,Y))\mathcal{Q}^\frac{1}{2}\|_{L_2(U,\mathbb H)}\to0,
\end{equation}
\begin{equation}\label{eq3}
\sup\limits_{|t|\leq l,|Y|\leq r}\int_{|x|_U<1}|F(t+t_n,Y,x)-\widetilde{F}(t,Y,x)|^2\nu({\rm d }x)\to0,
\end{equation}
\begin{equation}\label{eq4}
\sup\limits_{|t|\leq l,|Y|\leq r}\int_{|x|_U\geq1}|G(t+t_n,Y,x)-\widetilde{G}(t,Y,x)|^2\nu({\rm d }x)\to0,
\end{equation}
as $n\to\infty$. As we have done in the proof of {Step 2}: let $\xi_n$ and $\widetilde\xi$ be the unique bounded solutions of
the shift equation and the limit equation respectively, and still denote $\phi_n=\xi_n-\widetilde{\xi}$.
We need to prove $\phi_n\to 0$ as $n\to\infty$ in the space $C(\mathbb{R},\mathcal{L}^2(\mathbf P;\mathbb H))$, i.e. for any $k>0$,
\begin{equation*}
\lim_{n\to\infty}\max_{|t|\leq k}\mathbb E|\phi_n(t)|^2=0.
\end{equation*}
Then we have $\xi_n(t)\to\widetilde\xi(t)$ in distribution uniformly in $t\in[-k,k]$ for any $k>0$. Since $\xi_n(t)$
and $\xi(t+t_n)$ share the same distribution, $\xi(t+t_n)\to\widetilde\xi(t)$ in distribution uniformly in
$t\in[-k,k]$ for all $k>0$. Then we have $\{t_n\}\in \mathfrak{\tilde{M}}_{\xi}$, and hence $\xi$ is strongly
compatible in distribution.

Since $\phi_n$ is the unique bounded solution of (\ref{4.12}), by Cauchy-Schwarz inequality, It\^o's isometry
 and properties of integrals for Poisson random measures we have
\begin{align*}
\mathbb E|\phi_n(t)|^2\leq&\frac{4K^2}{\omega}\Bigg(\int_{-\infty}^t{\rm e}^{-\omega(t-\tau)}\mathbb E|a_n(\tau)-
\widetilde{a}(\tau)|^2{\rm d }\tau  \\
&\quad+\omega\int_{-\infty}^t{\rm e}^{-2\omega(t-\tau)}\mathbb E\|(b_n(\tau)-\widetilde{b}(\tau))
\mathcal{Q}^\frac{1}{2}\|_{L_2(U,\mathbb H)}^2{\rm d }\tau \\
&\quad+\omega\int_{-\infty}^t\int_{|x|_U<1}{\rm e}^{-2\omega(t-\tau)}\mathbb E|c_n(\tau,x)-\widetilde{c}(\tau,x)|^2\nu({\rm d }x){\rm d }\tau \\
&\quad+(2\omega+2b)\int_{-\infty}^t\int_{|x|_U\geq1}{\rm e}^{-\omega(t-\tau)}
\mathbb E|d_n(\tau,x)-\widetilde{d}(\tau,x)|^2\nu({\rm d }x){\rm d }\tau\Bigg).
\end{align*}
Furthermore, considering (\ref{4.14})-(\ref{4.15 3}) and ${\rm e}^{-2\omega(t-\tau)}\leq {\rm e}^{-\omega(t-\tau)}(t\geq\tau)$, we get
\begin{align*}
\mathbb E|\phi_n(t)|^2\leq&\int_{-\infty}^t{\rm e}^{-\omega(t-\tau)}
\bigg[\left(\frac{8K^2\mathcal{L}^2}{\omega}+32K^2\mathcal{L}^2
+\frac{16K^2\mathcal{L}^2b}{\omega}\right)\mathbb E|\phi_n(\tau)|^2+\frac{8K^2}{\omega}\mathbb{E}I_1^2(n,\tau)\\
&\quad+8K^2\mathbb EI_2^2(n,\tau)+8K^2\mathbb{E}I_3^2(n,\tau)
+\left(16K^2+\frac{16K^2b}{\omega}\right)\mathbb{E}I_4^2(n,\tau)\bigg]{\rm d}\tau.
\end{align*}
By Lemma \ref{lB} we have for $l>k>0$
\begin{align}\label{4.25}
\max_{|t|\leq k}\mathbb{E}|\phi_n(t)|^2\leq&
\frac{{\rm e}^{\alpha k}{\rm e}^{-\alpha l}}{\alpha}\sup_{t\in\mathbb{R}}\bigg[
\frac{8K^2}{\omega}\mathbb{E}I_1^2(n,t)+8K^2\mathbb{E}I_2^2(n,t)\\
&\quad+8K^2\mathbb{E}I_3^2(n,t)+\Big(16K^2+\frac{16K^2b}{\omega}\Big)\mathbb EI_4^2(n,t)\bigg]\nonumber\\
&+\frac{1-{\rm e}^{-\alpha k}{\rm e}^{-\alpha l}}{\alpha}\max_{|t|\leq l}\bigg[\frac{8K^2}{\omega}\mathbb{E}I_1^2(n,t)+8K^2\mathbb EI_2^2(n,t)\nonumber\\
&\quad+8K^2\mathbb{E}I_3^2(n,t)
+\Big(16K^2+\frac{16K^2b}{\omega}\Big)\mathbb{E}I_4^2(n,t)\bigg],\nonumber
\end{align}
where
$$
\alpha:=\omega-\bigg[\frac{8K^2\mathcal{L}^2}{\omega}+32K^2\mathcal{L}^2
+\frac{16K^2\mathcal{L}^2b}{\omega}\bigg]>0
$$
by assumption \eqref{L11}. By Remark \ref{H(f,g,F,G)}-(i), we have
\begin{align}\label{4.301}
\mathbb E I_1^2(n,t)&=\mathbb E|f^{t_n}(t,\widetilde{\xi}(t))-\widetilde{f}(t,\widetilde{\xi}(t))|^2 \\
&\leq2\mathbb E|f^{t_n}(t,\widetilde{\xi}(t))|^2+2\mathbb E|\widetilde{f}(t,\widetilde{\xi}(t))|^2
\leq4(A_0+\mathcal Lr)^2,\nonumber
\end{align}
for any $t\in\mathbb R$. Similarly, we have
\begin{align}\label{4.302}
\mathbb E I_2^2(n,t)\leq 4(A_0+\mathcal Lr)^2,
\end{align}
\begin{align}\label{4.303}
\mathbb E I_3^2(n,t)\leq 4(A_0+\mathcal Lr)^2,
\end{align}
\begin{align}\label{4.304}
\mathbb E I_4^2(n,t)\leq 4(A_0+\mathcal Lr)^2,
\end{align}
for any $t\in\mathbb R$, $n\in \mathbb N$.
  Let $\{l_n\}$ be a sequence of positive numbers such that $l_n\to{+\infty}$ as $n\to\infty$. According
to (\ref{4.301})-(\ref{4.304}) and (\ref{4.25}), we obtain
\begin{align}\label{4.26}
\max_{|t|\leq k}\mathbb E|\phi_n(t)|^2\leq&\frac{32K^2{\rm e}^{\alpha k}{\rm e}^{-\alpha l_n}}{\alpha}\left(\frac{1}{\omega}
+4+\frac{2b}{\omega}\right)(A_0+\mathcal{L}r)^2   \\
&+\frac{8K^2(1-{\rm e}^{-\alpha k}{\rm e}^{-\alpha l_n})}{\alpha}\max_{|t|\leq l_n}\bigg[\frac{1}{\omega}
\mathbb EI_1^2(n,t)+\mathbb EI_2^2(n,t)  \nonumber\\
&\quad+\mathbb EI_3^2(n,t)+\left(2+\frac{2b}{\omega}\right)\mathbb EI_4^2(n,t)\bigg].\nonumber
\end{align}
 Thanks to Remark \ref{remD1}-2-(iii), (\ref{eq1})-(\ref{eq4}) and the uniform integrability of the families
 $\{I_i^2(n,\tau): n\in\mathbb N, \tau\in \mathbb R\}$ $(i=1,2,3,4)$, letting $n\to\infty$ in (\ref{4.26}) we get for any $k>0$
\begin{align*}
\lim_{n\to\infty}\max_{|t|\leq k}\mathbb E|\phi_n(t)|^2=0.
\end{align*}
The proof is complete.
\end{proof}

\begin{coro}\label{Co 4.7}
Consider \eqref{SemiSDE}. Assume that the conditions of \textsl{Theorem} {\rm\ref{thmtwo}} hold.
\begin{enumerate}
\item
If $f,g,F,G$ are jointly stationary (respectively,
$\tau$--periodic, quasi-periodic with the spectrum of frequencies
$\nu_1,\nu_2,\dots,\nu_m$, almost periodic, almost automorphic,
Birkhoff recurrent, Lagrange stable, Levitan almost periodic, almost
recurrent, Poisson stable) in $t\in\mathbb R$ uniformly w.r.t. $Y\in \mathbb H$ on every bounded subset, then so is the unique bounded
solution $\xi$ of {\rm(\ref{SemiSDE})} in distribution.

\item
If $f,g,F,G$ are jointly pseudo-periodic (respectively, pseudo-recurrent)
and $f,g,F,G$ are jointly Lagrange stable in $t\in\mathbb R$ uniformly w.r.t. $Y\in \mathbb H$ on every bounded subset, then the
unique bounded solution $\xi$ of {\rm(\ref{SemiSDE})} is pseudo-periodic
(respectively, pseudo-recurrent) in distribution.
\end{enumerate}
\end{coro}

\begin{proof}
This statement follows from Theorems \ref{th1}, \ref{thmtwo} and Remark \ref{remBUC}.
\end{proof}

\section{Stability of the bounded (Poisson stable) solution}

\begin{theorem}\label{thC2}
Consider {\rm(\ref{SemiSDE})}. Assume that $A$ generates a dissipative $\mathcal C^0$-semigroup such that
{\rm(\ref{expi})} holds and $f\in C(\mathbb{R}\times\mathbb H,\mathbb H)$, $g\in C(\mathbb{R}\times\mathbb H,L(U,\mathbb H))$,
$F,G\in C(\mathbb{R}\times\mathbb H,\mathcal L^2(\nu;\mathbb H))$. Suppose further that $f, g, F, G$ satisfy
{\rm(E1)}, {\rm(E2)} and $W$ and $N$ are the L\'evy-It\^o decomposition components of the two-sided L\'evy process $L$
with assumptions stated in Section \ref{levy}. If
\begin{equation}\label{lmin}
\mathcal L<\cfrac{\omega}{K\sqrt{5(1+4\omega+2b)}},
\end{equation}
then the unique solution $\xi \in C_b(\mathbb R, \mathcal L^2(\mathbf P;\mathbb H))$ of {\rm(\ref{SemiSDE})}
is globally asymptotically stable in the sense of square-mean.
\end{theorem}

\begin{proof}
Let $Y(t;t_0,Y_0)$ denote the solution of (\ref{SemiSDE}) passing through $Y_0$ at the initial time $t_0$.
Then for any $t\geq t_0$ and $Y_0$, $Y_1\in\mathcal L^2(\mathbf P;\mathbb H)$, we have for $i=0,1$
\begin{align*}
Y(t;t_0,Y_i)=&T(t-t_0)Y_i+\int_{t_0}^tT(t-s)f(s,Y(s;t_0,Y_i)){\rm d }s
+\int_{t_0}^tT(t-s)g(s,Y(s;t_0,Y_i)){\rm d }W(s)\\
&+\int_{t_0}^t\int_{|x|_U<1}T(t-s)F(s,Y(s-;t_0,Y_i),x)\widetilde{N}({\rm d }s,{\rm d }x)\\
&+\int_{t_0}^t\int_{|x|_U\geq1}T(t-s)G(s,Y(s-;t_0,Y_i),x)N({\rm d }s,{\rm d }x).
\end{align*}
By Cauchy-Schwarz inequality, It\^o's isometry and properties of integrals for Poisson random measures, we have
\begin{align}\label{5.7}
\mathbb{E}&|Y(t;t_0,Y_0)-Y(t;t_0,Y_1)|^2 \\
&\leq5\mathbb{E}|T(t-t_0)(Y_0-Y_1)|^2
+5\mathbb{E}\left|\int_{t_0}^tT(t-s)(f(s,Y(s;t_0,Y_0))-f(s,Y(s;t_0,Y_1))){\rm d }s\right|^2 \nonumber\\
&\quad+5\mathbb{E}\left|\int_{t_0}^tT(t-s)(g(s,Y(s;t_0,Y_0))-g(s,Y(s;t_0,Y_1))){\rm d }W(s)\right|^2 \nonumber\\
&\quad+5\mathbb{E}\bigg|\int_{t_0}^t\int_{|x|_U<1}T(t-s)(F(s,Y(s-;t_0,Y_0),x)-F(s,Y(s-;t_0,Y_1),x))
\widetilde{N}({\rm d }s,{\rm d }x)\bigg|^2 \nonumber\\
&\quad+5\mathbb{E}\bigg|\int_{t_0}^t\int_{|x|_U\geq1}T(t-s)(G(s,Y(s-;t_0,Y_0),x)-G(s,Y(s-;t_0,Y_1),x))
N({\rm d }s,{\rm d }x)\bigg|^2  \nonumber\\
&\leq 5K^2{\rm e}^{-2\omega(t-t_0)}\mathbb E|Y_0-Y_1|^2 \nonumber\\
&\quad+5K^2\int_{t_0}^t{\rm e}^{-\omega(t-s)}{\rm d }s\cdot\int_{t_0}^t{\rm e}^{-\omega(t-s)}\mathbb E
|f(s,Y(s;t_0,Y_0))-f(s,Y(s;t_0,Y_1))|^2{\rm d }s \nonumber\\
&\quad+5K^2\int_{t_0}^t{\rm e}^{-2\omega(t-s)}\mathbb{E}\|(g(s,Y(s;t_0,Y_0))-g(s,Y(s;t_0,Y_1)))\mathcal{Q}^
\frac{1}{2}\|_{L_2(U,\mathbb H)}^2{\rm d }s  \nonumber \\
&\quad+5K^2\int_{t_0}^t\int_{|x|_U<1}{\rm e}^{-2\omega(t-s)}\mathbb E|F(s,Y(s-;t_0,Y_0),x)-F(s,Y(s-;t_0,Y_1),x)|^2
\nu({\rm d }x){\rm d }s  \nonumber\\
&\quad+\bigg(10K^2\int_{t_0}^t\int_{|x|_U\geq1}{\rm e}^{-2\omega(t-s)}\mathbb E|G(s,Y(s-;t_0,Y_0),x)
-G(s,Y(s-;t_0,Y_1),x)|^2 \nu({\rm d }x){\rm d }s  \nonumber\\
&\qquad+10K^2\int_{t_0}^t\int_{|x|_U\geq1}{\rm e}^{-\omega(t-s)}\nu({\rm d }x){\rm d }s \nonumber \\
&\quad\qquad\cdot\int_{t_0}^t\int_{|x|_U\geq1}{\rm e}^{-\omega(t-s)}\mathbb E|G(s,Y(s-;t_0,Y_0),x)
-G(s,Y(s-;t_0,Y_1),x)|^2\nu({\rm d }x){\rm d }s\bigg)\nonumber\\
&\leq5K^2{\rm e}^{-\omega(t-t_0)}\mathbb E|Y_0-Y_1|^2  \nonumber\\
&\quad+5\left(\frac{1}{\omega}+4+\frac{2b}{\omega}\right)K^2\mathcal{L}^2\int_{t_0}^t
{\rm e}^{-\omega(t-s)}\mathbb E|Y(s;t_0,Y_0)-Y(s;t_0,Y_1)|^2{\rm d }s. \nonumber
\end{align}
Set $U(t):={\rm e}^{\omega t}\mathbb E|Y(s;t_0,Y_0)-Y(s;t_0,Y_1)|^2$ for $t\geq t_0$. It follows from (\ref{5.7}) that
\begin{align}\label{5.8}
U(t)\leq5K^2{\rm e}^{\omega t_0}\mathbb E|Y_0-Y_1|^2
+5\left(\frac{1}{\omega}+4+\frac{2b}{\omega}\right)K^2\mathcal{L}^2\int_{t_0}^t
U(s){\rm d }s.
\end{align}
Along with (\ref{5.8}) we consider the equation for $t\ge t_0$
\begin{equation*}
V(t)=5K^2{\rm e}^{\omega t_0}\mathbb E|Y_0-Y_1|^2
+5\left(\frac{1}{\omega}+4+\frac{2b}{\omega}\right)K^2\mathcal{L}^2\int_{t_0}^t
V(s){\rm d }s;
\end{equation*}
that is, $V$ satisfies the following Cauchy problem
\begin{equation*}
V'(t)=5\left(\frac{1}{\omega}+4+\frac{2b}{\omega}\right)K^2\mathcal L^2 V(t),\quad V(t_0)=5K^2{\rm e}^{\omega t_0}\mathbb E|Y_0-Y_1|^2.
\end{equation*}
Solve the above equation and we obtain
\begin{equation*}
V(t)=5K^2{\rm e}^{\omega t_0}\mathbb E|Y_0-Y_1|^2\exp\bigg\{5\left(\frac{1}{\omega}+4+\frac{2b}{\omega}\right)K^2\mathcal{L}^2
(t-t_0)\bigg\}.
\end{equation*}
The comparison principle implies that $U(t)\leq V(t)$, i.e.
\begin{equation*}
U(t)\leq5K^2{\rm e}^{\omega t_0}\mathbb E|Y_0-Y_1|^2\exp\bigg\{5\left(\frac{1}{\omega}+4+\frac{2b}{\omega}\right)K^2\mathcal{L}^2
(t-t_0)\bigg\}, \quad\hbox{for } t\geq t_0.
\end{equation*}
So by the definition of $U(t)$ we get
\begin{align}\label{Y1}
\mathbb E&|Y(t;t_0,Y_0)-Y(t;t_0,Y_1)|^2 \\
&\leq5K^2\mathbb E|Y_0-Y_1|^2\exp\bigg\{-\left[\omega
-5\left(\frac{1}{\omega}+4+\frac{2b}{\omega}\right)K^2\mathcal{L}^2\right](t-t_0)\bigg\}
 \nonumber
\end{align}
for any $t\geq t_0$.
By Theorem \ref{thmtwo}, if the Lipschitz constant $\mathcal L$ satisfies (\ref{lmin}),
then (\ref{SemiSDE}) has a unique $\mathcal L^2$-bounded solution $\xi$. Let $\xi(t)=Y(t;t_0,\xi(t_0))$.
It follows from \eqref{Y1} that for any $t\geq t_0$ and $Y_0\in\mathcal L^2(\mathbf P;\mathbb H)$,
\begin{equation*}
\mathbb E|Y(t;t_0,Y_0)-\xi(t)|^2
\leq5K^2\mathbb E|Y_0-\xi(t_0)|^2\exp\bigg\{-\left[\omega
-5\left(\frac{1}{\omega}+4+\frac{2b}{\omega}\right)K^2\mathcal{L}^2\right](t-t_0)\bigg\}.
\end{equation*}
Hence by the assumption of $\mathcal L$ in (\ref{lmin}), we have for any $Y_0\in\mathcal L^2(\mathbf P;\mathbb H)$
\begin{equation}\label{con}
\lim_{t\to\infty}\mathbb E|Y(t;t_0,Y_0)-\xi(t)|^2=0.
\end{equation}
So the unique $\mathcal L^2$-bounded solution is globally asymptotically stable in square-mean sense.
\end{proof}

\begin{coro}\label{Co convergence}
Under the conditions of \textsl{Theorem} {\rm \ref{thC2}}, \eqref{SemiSDE} has a unique
$\mathcal L^2$-bounded solution $\xi\in C(\mathbb R,B[0,r])$. According to {\rm(\ref{con})},
we have for any $Y_0\in\mathcal L^2(\mathbf P;\mathbb H)$
\begin{equation*}
\limsup_{t\to\infty}\mathbb E|Y(t;t_0,Y_0)|^2<r+1,
\end{equation*}
where $r$ is given by \eqref{four 1}.
That is to say, all solutions of \eqref{SemiSDE} with initial value belonging to $\mathcal L^2(\mathbf P;\mathbb H)$
are bounded by a constant after sufficiently long time.
\end{coro}

\section{Applications}
In this section, we illustrate our theoretical results by two examples.

\begin{example}\rm
Let us consider an ordinary differential equation driven by a two-sided L\'evy noise:
\begin{align}\label{6.1}
{\rm d }y&=\left(-4y+\frac{1}{8}y(\sin t+\cos\sqrt 3t)\right){\rm d}t+\frac{1}{5}y\cos\left(\cfrac{1}{2+\sin t+\sin \sqrt 2t}\right){\rm d} W \\
&\quad+\int_{|x|_U<1}\frac{1}{5}y\widetilde N({\rm d}t,{\rm d}x)
+\int_{|x|_U\geq1}\frac{1}{4}y\sin\left(\cfrac{1}{3+\cos t+\cos{\pi t}}\right) N({\rm d}t,{\rm d}x)  \nonumber \\
&=:\left(Ay+f(t,y)\right){\rm d}t+g(t,y){\rm d}W+\int_{|x|_U<1}F(t,y,x)\widetilde N({\rm d}t,{\rm d}x)
+\int_{|x|_U\geq1}G(t,y,x)N({\rm d}t,{\rm d}x),\nonumber
\end{align}
where $W$ is a one-dimensional two-sided Brownian motion and $N$ is a Poisson random measure in $\mathbb{R}$
independent of $W$. Since $f,g,F,G$ are respectively quasi-periodic, Levitan almost periodic, stationary
and almost automorphic in $t$, uniformly w.r.t $y$ on any bounded subset of $\mathbb R$,
it follows that $f, g, F, G$ are jointly Levitan almost periodic. It is clear that
$A$ generates a dissipative semigroup on $\mathbb R$ with $K=1$,
$\omega=4$. For any constant $A_0\geq0$, conditions {\rm(E1)} and {\rm(E1$'$)} always hold.
The Lipschitz constants of $f,g,F,G$ can be respectively chosen as $\frac{1}{4},\frac{1}{5},\frac{1}{5},\frac{1}{4}$,
so the Lipschitz conditions {\rm(E2)} and {\rm(E2$'$)}
are satisfied with $\mathcal L=\frac{1}{4}$, if
\[
\int_{|x|<1}\left(\frac{1}{5}\right)^p\nu({\rm d}x)\leq\left(\frac{1}{4}\right)^p \quad \hbox{and} \quad
\int_{|x|\geq 1}\left(\frac{1}{4}\right)^p\nu({\rm d}x)\leq\left(\frac{1}{4}\right)^p
\]
for $p=2$ and some constant $p>2$, i.e.
\[
\nu(-1,1)<\frac{25}{16} \quad \hbox{and} \quad b\leq 1.
\]
For the stochastic ordinary differential equation, condition {\rm(E3)} naturally holds.

Since coefficients $f,g,F,G$ satisfy both Lipschitz and global linear growth conditions, \eqref{6.1}
has a global in time solution. By Theorem \ref{thmtwo}, (\ref{6.1}) has a unique $\mathcal{L}^2$-bounded
solution provided $\nu(-1,1)<\frac{25}{16}$, $b\leq 1$. Conditions (\ref{L11}) and (\ref{lmin}) become
\[
\frac{1}{4}< \cfrac{4}{2\sqrt{2+32+4b}}  \quad \hbox{and} \quad
\frac{1}{4}<\cfrac{4}{\sqrt{5(1+16+2b)}},
\]
i.e. $b<\frac{15}{2}$ and $b<\frac{171}{10}$, respectively. By Corollary \ref{Co 4.7} the unique
$\mathcal{L}^2$-bounded solution is Levitan almost periodic in distribution and according to
Theorem \ref{thC2} this Levitan almost periodic solution is globally asymptotically stable in square-mean sense.
\end{example}

\begin{example}\rm
Consider the stochastic heat equation on the interval $[0,1]$ with Dirichlet boundary condition:
\begin{align}\label{6.2}
\frac{\partial u}{\partial t}(t,\xi)&=\frac{\partial^2 u}{\partial \xi^2}(t,\xi)+
\frac{1}{5}(\cos t+\sin{\sqrt2t})\sin u(t,\xi) \\
&\quad+u(t,\xi)\cdot\sin\left(\cfrac{1}{2+\cos t+\cos{\sqrt2t}}\right)\frac{\partial W}{\partial t}(t,\xi)
+\cfrac{\cos u(t,\xi)}{3(\sin\sqrt2t+2)}\frac{\partial Z}{\partial t}(t,\xi) \nonumber\\
&=:\frac{\partial^2u}{\partial\xi^2}(t,\xi)+f(t,u(t,\xi))+g(t,u(t,\xi))\frac{\partial W}{\partial t}(t,\xi)
+h(t,u(t,\xi))\frac{\partial Z}{\partial t}(t,\xi), \nonumber\\
u(t,0)&=u(t,1)=0,\qquad t>0,\quad \xi\in(0,1).  \nonumber
\end{align}
Here $W$ is a $\mathcal{Q}$-Wiener process on $\mathcal{L}^2(0,1)$ with ${\rm Tr}\mathcal{Q}<\infty$ and $Z$
is a L\'evy pure jump process on $\mathcal{L}^2(0,1)$ which is independent of $W$. Let $A$ be the Laplace
operator, then $A:D(A)=H^2(0,1)\cap H^1_0(0,1)\to\mathcal{L}^2(0,1)$. Denote $\mathbb H=U:=\mathcal{L}^2(0,1)$
and the norm on $\mathbb H$ by $\|\cdot\|$. Then we can write \eqref{6.2} as an abstract evolution equation
\begin{align}\label{6.3}
{\rm d }Y=&(AY+F(t,Y)){\rm d }t+G(t,Y){\rm d }W+\int_{|z|_U<1}H(t,Y,z)\widetilde{N}({\rm d }t,{\rm d }z)\\
&+\int_{|z|_U\geq1}H(t,Y,z)N({\rm d }t,{\rm d }z)\nonumber
\end{align}
on the Hilbert space $\mathbb H$, where
\begin{equation*}
Y:=u,  \quad F(t,Y):=f(t,u), \quad G(t,Y):=g(t,u),
\end{equation*}
\begin{equation*}
\int_{|z|_U<1}H(t,Y,z)\widetilde{N}({\rm d }t,{\rm d }z)+\int_{|z|_U\geq1}H(t,Y,z)N({\rm d }t,{\rm d }z):=h(t,u){\rm d}Z
\end{equation*}
with
\begin{align*}
Z(t,\xi)=\int_{|z|_U<1}z\widetilde{N}(t,{\rm d }z)+\int_{|z|_U\geq1}zN(t,{\rm d }z),\quad H(t,Y,z)=h(t,u)z.
\end{align*}
By the L\'evy-It\^o decomposition theorem, we assume that the L\'evy pure jump process $Z$ on $\mathcal{L}^2(0,1)$ is decomposed
as above. Note that the operator $A$ has eigenvalues $\{-n^2\pi^2\}_{n=1}^\infty$ and generates a
$\mathcal C^0$-semigroup $\{T(t)\}_{t\geq 0}$ on $\mathbb H$ satisfying
$\|T(t)\|\leq {\rm e}^{-\pi^2t}$ for $t\geq0$, i.e. $K=1$ and $\omega=\pi^2$.
Conditions (E1) and {\rm(E1$'$)} hold for any $A_0\geq 0$.
The Lipschitz constants of $f, g, h$ can be respectively chosen as $\frac{2}{5}$, $1$, $\frac{1}{3}$,
so conditions (E2) and {\rm(E2$'$)} are satisfied with the Lipschitz constant $\mathcal L=\max\Big\{\frac{2}{5},\|\mathcal Q^\frac{1}{2}\|_{L(U,U)},\frac{1}{3}(\nu(B_1(0)))^\frac{1}{p},\frac{1}{3}b^\frac{1}{p}\Big\}$,
where $B_1(0)$ is the open ball in $U$ centered at the origin with radius $1$
and $p$ is the constant in {\rm(E1$'$)} and {\rm(E2$'$)}.
Since $f, h$ are bounded and $g$ is linear, condition (E3) holds.
The restrictions of (\ref{L11}) and (\ref{lmin}) respectively become
\begin{equation}\label{b}
\mathcal L< \cfrac{\pi^2}{2\sqrt{2+8\pi^2+4b}}   \quad  \hbox{and} \quad
\mathcal L< \cfrac{\pi^2}{\sqrt{5(1+4\pi^2+2b)}}.
\end{equation}

Note that $F$ is quasi-periodic in $t$ and $H$ is periodic in $t$ uniformly w.r.t. $Y\in\mathbb H$;
$G$ is Levitan almost periodic in $t$ uniformly w.r.t. $Y$ on any bounded subset of $H$.
By Theorem \ref{thmtwo}, (6.3) (i.e. (6.2)) admits a unique $\mathcal L^2$-bounded
solution. It follows from Corollary \ref{Co 4.7} that this unique bounded solution
is Levitan almost periodic in distribution. By Theorem \ref{thC2}, this bounded solution
is globally asymptotically stable in square-mean sense. By Corollary \ref{Co convergence},
all solutions of (\ref{6.2}) with $\mathcal L^2$-initial value are bounded by a constant
after sufficiently long time.
\end{example}

\section*{Acknowledgements}
This work is partially supported by NSFC Grants 11522104, 11871132, 11925102,
and Xinghai Jieqing and DUT19TD14 funds from Dalian University of Technology.

\end{document}